\documentclass[12pt]{article}
\usepackage[utf8]{inputenc}  
\usepackage{csquotes}
\usepackage[T1]{fontenc}
\usepackage{graphicx}
\usepackage{tikz}
\usetikzlibrary{positioning}
\usepackage{float}
\usepackage{algorithm}
\usepackage{algorithmic}
\usepackage[utf8]{inputenc}
\usepackage[export]{adjustbox}
\usepackage{wrapfig}
\usepackage{amsfonts}
\usepackage{amsthm}
\usepackage{array}
\usepackage[top=2cm, bottom=2cm, left=2cm, right=2cm]{geometry}
\usepackage{amssymb}
\usepackage{xcolor}
\usepackage{mathtools}
\usepackage{comment}
\usepackage[main=english]{babel}
\usepackage{subcaption}
\usepackage{dsfont}
\usepackage{bbm}
\usepackage{lmodern}
\usepackage{indentfirst}
\usepackage{tabto}
\usepackage{color}
\usepackage[nottoc]{tocbibind}
\usepackage{systeme}
\usepackage{eso-pic}
\usepackage{float}
\usepackage{titlesec}
\usepackage{fancyhdr}
\usepackage{appendix}
\usepackage{listings}
\usepackage{pgf, tikz}
\usepackage[utf8]{inputenc}
\usepackage[T1]{fontenc}
\usepackage{lastpage}

\definecolor{bf}{rgb}{0,0,0.6} 
\definecolor{mygray}{gray}{0.85}
\definecolor{darkWhite}{rgb}{0.94,0.94,0.94}

\graphicspath{{images/}}

\newcommand{\esp}[1]{\mathbb{E}\mathopen{}\left[#1\right]}

\newcommand{\reels}{\mathbb{R}}
\newcommand{\entiers}{\mathbb{N}}

\usepackage{thmtools,thm-restate}

\newcommand{\FuncDef}[4]{\ensuremath{\left\{\begin{array}{c} #1\longrightarrow #2\\ #3\longrightarrow #4\\\end{array}\right.}}

\newcommand{\R}{\mathbb{R}}

\title{Noise through an additional variable for mean field games master equation on finite state space}
\author{Charles Bertucci, Charles Meynard}
\usepackage[backend=bibtex,natbib=true,style=numeric,maxbibnames=99]{biblatex}
\renewbibmacro{in:}{}
\usepackage{hyperref}

\numberwithin{equation}{section}

\usepackage{amsthm, amssymb}

\DeclareNameAlias{default}{family-given}

\newtheorem{thm}{Theorem}[section]
\newtheorem{definition}[thm]{Definition}
\newtheorem{prop}[thm]{Proposition}
\newtheorem{lemma}[thm]{Lemma}
\newtheorem{corol}[thm]{Corollary}
\theoremstyle{definition}
\newtheorem{remarque}[thm]{Remark}
\newtheorem{exemple}[thm]{Example}
\newtheorem{hyp}[thm]{Hypothesis}

\newenvironment{acknowledgements} {\begin{abstract}} {\end{abstract}}
\begin{document}
\maketitle
\begin{abstract}
     This paper provides a mathematical study of the well-posedness of master equation on finite state space involving terms modelling common noise. In this setting, the solution of the master equation depends on an additional variable modelling the value of a stochastic process impacting all players.
     Using technique from viscosity solutions, we give sufficient conditions for the existence of a Lipschitz continuous solution on any time interval. Under some structural assumptions, we are even able to treat cases in which the dynamics of this stochastic process depend on the state of the game.
\end{abstract}
\setcounter{tocdepth}{1}
\tableofcontents
\section{Introduction}
This paper is the first of a series devoted to the systematic study of mean field games (MFG) master equations associated to games in which a common source of randomness (or noise) is affecting all players through a finite dimensional parameter. We here treat the case of a master equation set on a subset of $\reels^d$. The case of a master equation in infinite dimensions shall be tackled in a forthcoming work. 

In a deterministic setting, the typical form of a master equation associated to a mean field game in a finite state space is 
\begin{equation}
\label{MFG sans bruit}
\left\{
\begin{array}{c}
-\partial_t V+(F(x,V)\cdot \nabla_x )V=G(x,V) \text{ for } (t,x)\in (0,T)\times\Omega, \\
V(T,x)=U_0(x) \text{ for } x \in \Omega.
\end{array}
\right.
\end{equation}
In this framework, $T>0$ is the horizon of the game, and $\Omega$ a bounded subset of $\reels^d$, $d \geq 1$. We are solving the previous equation for $V:[0,T]\times \Omega\to \reels^d$, given $G,F:\Omega\times \reels^d\to\reels^d$ which describe respectively the evolution of the value and of the mean field quantities, and $U_0:\Omega\to \reels^d$ is the final value. The solution $V$ of the master equation shall be called the value of the game.

The equation \eqref{MFG sans bruit} is naturally associated to the system of forward backward ordinary differential equations set on $[t_0,T]$.

\begin{equation}
\label{caractéristique MFG}
\left\{
\begin{array}{c}
     \frac{dX_t}{dt}=-F(X_t,U_t),  \\
     \frac{dU_t}{dt}=-G(X_t,U_t),\\
     X_{t_0}=x,\\
     U_T=U_0(X_T).\\
\end{array}
\right.
\end{equation} 
Given a solution of $V$ of \eqref{MFG sans bruit} and $(X_t,U_t)_{t \in [t_0,T]}$ of \eqref{caractéristique MFG}, we get that for all $t\in [t_0,T], V(t,X_t) = U_t$. Hence, the existence of a solution (resp. regularity) to \eqref{MFG sans bruit} is linked to uniqueness (resp. stability) properties of \eqref{caractéristique MFG}. At this level, it might seem strange that a forward-backward structure is imposed on \eqref{caractéristique MFG}, instead of a forward-forward or backward-backward one for instance. This is mainly due to the fact that the interpretation of \eqref{caractéristique MFG} is natural in terms of the underlying game.

In this paper, we aim at investigating situations in which the characteristics are in fact stochastic, i.e. when \eqref{caractéristique MFG} is a forward-backward stochastic differential equation (SDE). Namely, we are going to consider that the couplings $G$ and $F$ are going to depend on the value of a stochastic process $(p_t)_{t\geq 0}$ defined by 
\[
\begin{array}{c}
     dp_t=-b(p_t)dt+\sqrt{2\sigma} dW_t,
\end{array}
\]
where $(W_t)_{t \geq 0}$ is a $m$-dimensional Brownian motion on a standard (fixed) probability space, $b : \R^m \to \R^m$ is a data of the problem and $\sigma \geq 0$ is also given. In such a case, we look for a value which also depends on the value of this stochastic process $(p_t)$. Thus the problem now becomes the one of finding a solution $V:[0,T]\times \Omega\times \R^m\to \R^d$ of 
\begin{equation}
\label{MFG backward}
\left\{
\begin{array}{c}
-\partial_t V+(F(x,p,V)\cdot \nabla_x) V+b(p)\cdot \nabla_p V-\sigma \Delta_p V=G(x,p,V) \text{ in } (0,T)\times \Omega \times \R^m, \\
V(T,x,p)=U_0(x,p) \text{ in } \Omega\times \mathbb{\reels}^m.
\end{array}
\right.
\end{equation}
This is what we mean by "noise through an additional variable": the presence of this new variable $p$ in the master equation is a consequence of the dependency of the coefficients on the noise process $(p_t)_{t\geq 0}$. This is a different approach to noise (or randomness) in mean field game from the usual additive common noise, studied for instance in \cite{cardaliaguet2015master}. We believe this modeling of common noise to be natural in a wide variety of applications, see \cite{bertucci2020mean,cardaliaguet:hal-01389128,quentinpetit} for examples.

In what follows we are going to analyze \eqref{MFG backward} in two different contexts. The first one is mainly the one we just presented, in which the dynamics of $b$ are purely exogenous to the underlying game. Taking proper care of the non-linearity $(F(x,U)\cdot \nabla_x) U$, we are able to extend the result existing on \eqref{MFG sans bruit} to the case of \eqref{MFG backward}. 

We then proceed to study cases in which we allow the vector field $b$, which drives the evolution of the stochastic process $(p_t)_{t\geq 0}$ to depend on the state of the mean fields and on the value of the game, i.e. we allow $b$ to depend on $x$ and $V$. To the extent of our knowledge, this is the first attempt at studying MFG in which the dynamics of the noise depend on the MFG. This kind of dependence is extremely natural. Indeed, consider for instance the case of $(p_t)_{t\geq 0}$ modeling environmental variables, we are then modeling the fact that the players can impact on these environmental variables. Consider also the case in which $(p_t)_{t\geq 0}$ stands for certain prices affecting a financial system, or indexes (such as GDP) influencing an economy. We are then also modeling the facts that the players can have an impact on those quantities, which seems particularly natural.

\paragraph{Bibliographical comments}
Mean field games were first introduced by Lasry and Lions in \cite{MFGjpmath,P_L_Lions}. Since then, many works have contributed to the field, both extending it with new modelling concerns and contributing to pushing the mathematical analysis of these models. The master equation approach was introduced in \cite{P_L_Lions} and lead to new mathematical questions. In general, the study of the master equation is done under structural assumptions, often on the monotonicity of the couplings $G$ and $F$ or the fact that we can integrate the master equation into an Hamilton-Jacobi-Bellman (HJB) equation (the so-called potential regime). A first well-posedness theory was obtained in \cite{cardaliaguet2015master} for classical solutions. Since then, several notions of weak solutions have been introduced, we can cite for instance the notion of monotone solution \cite{JEP_2021__8__1099_0,bertucci2023monotone,doi:10.1137/21M1450008} which allows to define a concept of solution of master equation for values which are merely continuous. Note that the type of noise we are interested in was also discussed in \cite{bertucci2023monotone}. Lipschitz solutions have been introduced in \cite{bertucci2023lipschitz} and allow to work with well-defined notion of solution up to an eventual explosion time. Other regimes of well-posedness were studied in the works \cite{mou2022wellposedness,gangbo2022mean}.

Concerning master equations on a finite state space, since the initial findings presented in \cite{P_L_Lions}, a lot of results have been obtained. Without being exhaustive, we can mention that two different types of common noise were introduced and studied in \cite{bertucci2018remarks,BAYRAKTAR202198}. A regularizing effect was established in \cite{bertucci2021master} under strong monotonicity assumptions. The paper \cite{cecchin2022weak} made the link between the master equation and the associated HJB equation in the potential case and \cite{bertucci2022mean} dealt with the convergence of the master equation in a finite state toward the master equation associated with a continuous state space when the number of states goes to infinity. Finally \cite{lions2023linear} explored existence and non uniqueness of solutions with irregular coefficients. 
None of these papers addresses the type of noise we propose to model here. Nonetheless, we shall follow an approach in the spirit of \cite{bertucci2023lipschitz,P_L_Lions}. Similar problems were encountered in the study of forward-backward stochastic differential equations (FBDSE) \cite{doi:10.1137/S0363012996313549}, we come back on the link between their approach and ours at the end of the paper.

\paragraph{Modelling comments}
Throughout this paper, we are not going to enter too much in details on the MFG interpretation of the master equation. While the main reason for this choice is that we focus here on the mathematical analysis of \eqref{MFG backward}, we believe important to mention another reason for this. In general, the class of MFG master equation naturally models situations which are not proper MFG, but general economic equilibria of large populations. This is for instance the case in \cite{bertucci2020mean,achdou2022class,bertucci2023singular}. We can also mention \cite{lions2021extended} which explains why the class of so-called extended MFG can naturally arise and lead to similar master equations. Hence, we shall not require here 
\[\Omega=\{x\in \left(\reels^+\right)^d \quad  \displaystyle\sum_i x_i=1\},\]
even if this set is extremely natural for finite state MFG. We refer to the appendix of \cite{JEP_2021__8__1099_0} for an explanation of how we can pass from our framework to this one.

\paragraph{Organization of the paper.}

In Section \ref{sec:autonomous}, we study the master equation when the noise process is solution of an autonomous stochastic differential equation. We start by establishing an a priori estimate and then prove results of existence and uniqueness of a global in time Lipschitz solution of the master equation.

In Section \ref{sec:non-autonomous}, we assume the dynamics of the noise may depend on the distribution of players and the value function of the game. We start by explaining the main difficulties arising from this dependence. We then proceed to justify some structural estimates on the drift and prove results of existence and uniqueness of a global Lipschitz solution in this case as well.

\section{The case of an autonomous noise process}\label{sec:autonomous}

In this section, we assume $b$ is a function of $p$ only. In other words, we are interested in the well-posedness of 

\begin{equation}
\label{MFG forward}
\left\{
\begin{array}{c}
\partial_t U+F(x,p,U)\cdot \nabla_x U+b(p)\cdot \nabla_p U-\sigma \Delta_p U=G(x,p,U) \text{ for }  t\in(0,T), x \in \Omega,p \in\R^m, \\
U(0,x,p)=U_0(x,p) \text{ for } (x,p)\in \Omega\times \reels^m.
\end{array}
\right. 
\end{equation}

Note that we have reversed time compared to \eqref{MFG sans bruit}, mainly to lighten notations. Recall that the unknown of the previous equation is a function $U : [0,T]\times\Omega\times \R^m \to \R^d$.
The functions $(F,G,b)$ are always assumed to be Lipschitz continuous (at least locally) in all the variables. We further assume that $\Omega$ is a smooth bounded domain of $\reels^d$ and that $F$ satisfies the following condition on its boundary:
\begin{equation}
\forall x\in \partial \Omega, (u,p)\in \reels^d\times \reels^m \quad \eta(x)\cdot F(x,p,u)\geq 0 ,
\end{equation}
where $\eta(x)$ denotes the normal vector to $\partial\Omega$ in $x$ directed toward the outside of the domain. This condition ensures we do not need to concern ourselves with boundary conditions for $U$ on $\partial\Omega$. Furthermore, this kind of assumptions is natural in the context of MFG.

\subsection{A propagation of monotonicity argument}
It is well known that nonlinear PDE of the form of \eqref{MFG forward} can create shocks or singularity in finite time, such as Burger's equation. In general, structural conditions are required to prevent such explosive phenomena. We refer to P.L. Lions lectures at Collège de France \cite{P_L_Lions} for extensive results concerning the PDE
\[\partial_t U+F(x,U)\cdot \nabla_x U=G(x,U) \text{ for } (t,x)\in (0,\infty)\times \R^d.\]
In particular, in a monotone regime, he showed, by a propagation of monotonicity argument, that a priori estimates can be obtained for such an equation. We shall say that we are in the monotone regime when the following Hypothesis holds.
\begin{hyp}
\label{hyp:monotone}
There exists $\alpha > 0$ such that $U_0$ is $\alpha-$monotone and $(G,F)$ is strongly monotone in $x$, i.e. for almost every $(x,u)$ in $\Omega\times \reels^d$, $\forall (\xi,\nu)\in\reels^{d}\times\reels^d$
\begin{equation}
    \xi^T (D_xU_0) \xi\geq \alpha |D_xU_0\cdot\xi|^2,
\end{equation}
 \begin{equation}
    \label{(F,G) monotone diff}
    (\xi,\nu)^T\left(\begin{array}{cc} D_x G& D_u G\\ D_x F & D_u F\end{array}\right) (\xi,\nu)\geq \alpha  |\xi|^2.
    \end{equation}
\end{hyp}

\begin{remarque}
It would be sufficient to require $(G,F)$ to be $\alpha-$monotone in $x$ instead of strong monotonicity. To avoid technical difficulties we here make a slightly stronger assumption as we delay such extension to a later section. In \cite{P_L_Lions}, other variants of the previous regime were also studied. The same approach can be carried on here. Indeed, we could also require $(G,F)$ to be $\alpha$-monotone with respect to $u$ instead of $x$, in which case the assumption on $U_0$ can be weakened. As this extension does not raise any particular difficulty, we leave it to the interested reader.
\end{remarque}

A similar approach is also valid in the case of \eqref{MFG forward} as the next result presents

\begin{prop}
\label{prop monotonicity}
    Assume that there exists $\alpha >0$ such that Hypothesis \ref{hyp:monotone} holds and let U be a smooth solution of \eqref{MFG forward}, uniformly Lipschitz in $(x,p)$.
    Then, for all $t \in [0,T], p \in \R^m$, $x \to U(t,x,p)$ is monotone. Furthermore, there exists a constant $C>0$ depending only on $F,G,b,U_0$ and $T$ such that 

    \[\forall t\leq T, \quad \|D_xU(t,\cdot)\|_\infty\leq C.\]
\end{prop}

\begin{proof}

\quad

\noindent \textit{An equation on the gradient}
\quad 

First of all, by writing $W=U\cdot \xi$ let us remark that $W$ is solution of 
\[\partial_t W+F\cdot\nabla_xW+b\cdot\nabla_p W-\sigma\Delta_p W=G\cdot \xi  .\]
Consider:
\begin{align*}
Z(t,x,p,\xi)&=\xi \cdot \nabla_x (U\cdot \xi)-\beta(t)|D_x U\cdot \xi|^2\\
&=\xi \cdot \nabla_x W-\beta(t)|\nabla_x W|^2 ,
\end{align*}
for a smooth function of time $\beta$ which has yet to be chosen. We can express the derivatives of $Z$ in terms of $(W,\beta)$:
\begin{equation*}
\left\{
\begin{array}{c}
\nabla_x Z=(\xi-2\beta(t)\nabla_x W) D^2_x W,\\
\nabla_p Z=(\xi-2\beta\nabla_x W)\cdot D^2_{px} W,\\
\Delta_p Z=(\xi -2\beta(t)\nabla_x W) \nabla_x \Delta_p W-2\beta(t)|D^2_{xp}W|^2,\\
\nabla_{\xi} Z=\nabla_x W+D_xU\cdot(\xi-2\beta\nabla_x W).
\end{array}
\right.
\end{equation*}
Since $U$ is smooth, we can derive directly the equation satisfied by $W$ to get the one satisfied by its gradient 
\begin{align*}
\partial_t(\nabla_x W)&=\nabla_x\left(G\cdot\xi-F\cdot \nabla_x W-b\cdot \nabla_p W+\sigma \Delta_p W\right)\\
&=\nabla_x (G\cdot\xi)-\nabla_x(F)\cdot \nabla_x W-F\cdot D^2_x W-b\cdot D^2_{xp} W+\sigma \Delta_p \nabla_x W .
\end{align*}
From which we get 
\begin{align*}
\partial_t Z&=-\beta'|\nabla_x W|^2+(\xi-2\beta\nabla_x W)\cdot \partial_t(\nabla_x W)\\
&=-\beta'|\nabla_x W|^2\\
&+(\xi-2\beta(t)\nabla_x W)\cdot \left(D_xG^T\xi-D_xF^T\nabla_xW\right)\\
&+(\xi-2\beta(t)\nabla_x W)\cdot\left(D_xU^TD_uG^T \xi-D_xU^T D_uF^T\nabla_x W\right)\\
&+(\xi-2\beta(t)\nabla_x W)\cdot\left(-FD^2_{x}W-bD^2_{xp} W+\sigma \Delta_p \nabla_x W\right).
\end{align*}
Remark that: 
\begin{align*}
    (\xi-2\beta(t)\nabla_x W)\cdot\left(D_xU^TD_uG^T \xi\right)&=\left(D_xU\cdot(\xi-2\beta(t)\nabla_x W)\right)\cdot\left(D_uG^T \xi\right)\\
    &=\left(\nabla_{\xi}Z-\nabla_x W\right)\cdot\left(D_uG^T \xi\right).
\end{align*}
and that a similar calculation may be performed for the term in $D_uF$.\\
We may also notice that
\begin{align*}
    (\xi-2\beta(t)\nabla_x W)\cdot\left(-FD^2_{x}W-bD^2_{xp} W+\sigma \Delta_p \nabla_x W\right)&=-F\cdot \nabla_xZ-b\cdot \nabla_p Z+\sigma\Delta_p Z+2\beta(t)\sigma |D^2_{xp}W|^2.
\end{align*}
Looking at the equation satisfied by $Z$ we get
\begin{align*}
\nonumber&\partial_t Z-\sigma\Delta_p Z+b\cdot \nabla_p Z+F\cdot \nabla_x Z+\left(\nabla_xWD_uF-\xi D_u G\right)\cdot \nabla_{\xi} Z\\
&=2\beta\sigma|D^2_{xp}W|^2-\frac{d\beta}{dt}|\nabla_x W|^2\\
\nonumber&+\xi D_xG\xi-\xi D_uG\nabla_xW-\nabla_x WD_xF\xi+\nabla_xWD_uF\nabla_xW\\
\nonumber&+2\beta\left(\nabla_x W D_x F \nabla_x W-\xi D_x G\nabla_x W\right) .
\end{align*}
We can now make use of the monotonicity of the couple $(G,F)$ by noticing that
\begin{align*}
    \xi D_xG\xi-\xi D_uG\nabla_xW-\nabla_x WD_xF\xi+\nabla_xWD_uF\nabla_xW&= \left(\begin{array}{c}-\xi\\ \nabla_x W\end{array}\right)^T\left(\begin{array}{cc} D_x G\quad D_u G\\ D_x F\quad D_u F\\\end{array}\right)\left(\begin{array}{c}-\xi\\ \nabla_x W\end{array}\right)\\
    &\geq \alpha |D_xG\xi|^2
\end{align*}
by assumption.
Now,

\begin{align*}
&\partial_t Z-\sigma\Delta_p Z+b\cdot \nabla_p Z+F\cdot \nabla_x Z+\left(\nabla_xWD_uF-\xi D_u G\right)\cdot \nabla_{\xi} Z\\
&\geq \alpha|D_xG\cdot \xi|^2-\frac{d\beta}{dt}|\nabla_x W|^2+2\beta\left(\nabla_xWD_xF\nabla_xW-\xi D_xG\nabla_xW\right)\\
&\geq (\alpha-\beta(t))|D_x G \cdot\xi|^2+\left(\beta(t)(\|D_xG\|^2+2\|D_xF\|)-\frac{d\beta}{dt}(t)\right)|\nabla_xW|^2.
\end{align*}
Let us take  $\beta_0<\alpha$ and $\beta(t)=\beta_0e^{-\lambda_\beta t}$ with $\lambda_\beta$ chosen in such fashion that 

\[\lambda_\beta\geq \lambda_\beta^*=\|D_xG\|^2+2\|D_xF\|.\]

\quad

By assumption on $U_0$ and for this well chosen $\beta$ we have:

\begin{equation}
\label{eq Z smooth }
\left\{
\begin{array}{c}
Z|_{t=0}\geq 0,\\
  \partial_t Z-\sigma\Delta_p Z+b\cdot \nabla_p Z+F\cdot \nabla_x Z+\left(\nabla_xWD_uF-\xi D_u G\right)\cdot \nabla_{\xi} Z\geq 0  \text{ for } (0,T)\times \Omega\times \reels^m.\\
\end{array}
\right. 
\end{equation}

\quad

\noindent\textit{$Z$ is non-negative}
\quad 
The non-negativity of $Z$ is standard and follows classical comparison results. We recall the main argument for the sake of completeness.
Recall that $U$ is Lipschitz and $D_xU$ is hence bounded. As a consequence the SDE
\begin{equation*}
\left\{
\begin{array}{l}
     \displaystyle p_t=p_0-\int_0^t b(p_s)ds+\sqrt{2\sigma}W_t, \\
     \displaystyle X_t=x-\int_0^t F(X_s,p_s,U(\theta-s,X_s,p_s))ds,\\
     \displaystyle \xi_t=\xi_0-\int_0^t \left(D_xU(\theta-s,X_s,p_s)\cdot\xi_s D_uF(X_s,p_s,U(\theta-s,X_s,p_s))-\xi_sD_uG(X_s,p_s,U(\theta-s,X_s,p_s))\right)ds.
\end{array}
\right.
\end{equation*}
are well defined on $[0,\theta]$ for any $\theta<T$. Since Z has at most quadratic growth in $\xi$ uniformly in $(x,p)$, applying Ito's Lemma to 
\[Z(\theta,x,p_0,\xi_0)-Z(0,X_{\theta},p_{\theta},\xi_{\theta})\]
and using \eqref{eq Z smooth } yields 
\[\forall t<T \quad \forall (x,p_0,\xi_0)\in \Omega \times \reels^m\times \reels^d \quad Z(t,x,p_0,\xi_0)\geq \esp{Z_0(X_t,p_t,\xi_t)}\geq 0\]

\noindent\textit{Estimate on the gradient}
\quad 
Of course, the non-negativity of Z on $[0,T]$ gives an estimate on the Lipschitz norm of $U$ in $x$:
\begin{align*}
    Z(t,x,p,\xi)\geq 0 &\iff \xi \cdot D_xU(t,x,p) \cdot \xi \geq \beta(t) \|D_xU(t,x,p)\cdot \xi\|^2\\
    &\implies \frac{\|D_xU(t,x,p)\cdot \xi\|}{\|\xi\|}\leq \frac{1}{\beta(t)} .
\end{align*}
Taking the supremum over all $(\xi,x,p)$ yields the required estimate:

\[\forall t\leq T, \quad \|D_xU(t,\cdot)\|_\infty\leq \frac{1}{\beta(t)}\leq \frac{2}{\alpha}e^{\lambda_\beta T}.\]
The monotonicity of $U$ in $x$ is a trivial consequence of $Z\geq 0$.

\end{proof}
We shall see later on that this estimate is crucial to obtain the well-posedness of \eqref{MFG forward}.


\subsection{On Lipschitz solutions}
Lipschitz solutions are weak solutions of the master equation. The main idea here is that should the vector fields $F\circ U$ and $G\circ U$ be known and Lipschitz continuous, finding the solution of the system reduce to a simple application of the Feynman-Kac formula. As a consequence we can define the solution of the master equation as a fixed point of a Feynman-Kac inspired functional. The main advantage of this approach is that it allows to define a concept of solution of \eqref{MFG forward} for solutions which are merely Lipschitz continuous. We refer the reader to the paper \cite{bertucci2023lipschitz} for a global presentation of this notion. This notion was only stated in the absence of the noise process $p$. Here, we only indicate how to adapt the results of \cite{bertucci2023lipschitz} to the case of \eqref{MFG forward} as the extension is quite immediate.

\quad

Consider two vector fields $A:\FuncDef{\reels^{+*}\times\reels^d\times\reels^m}{\reels^d}{(t,x,p)}{A(t,x,p)}$ and $B:\FuncDef{\reels^{+*}\times\reels^d\times\reels^m}{\reels^d}{(t,x,p)}{B(t,x,p)}$ Lipschitz in $(x,p)$ uniformly in $t$, such that $\Omega$ is left invariant by $A$, an initial condition $U_0$ Lipschitz in $(x,p)$ and a Lipschitz continuous $b:\reels^m\to\reels^m$. 
Let $(\Omega_\mathcal{P},\mathcal{F},\mathbb{P})$ be a sufficiently rich probability space, we define the functional $\psi$ by $\psi(T,A,B,U_0)=V$ if 

\[\left\{
\begin{array}{c}
 \forall t\in[0,T), (x,p)\in\Omega\times\reels^m \quad V(t,x,p)=\displaystyle\mathbb{E}\left[\left.\int_0^{t}A(t-s,X_s,p_s)ds+U_0(X_t,p_t)\right|X_0=x,p_0=p\right],\\
dX_s=-B(t-s,X_s,p_s)dt,\\
dp_s=-b(p_s)dt+\sqrt{2\sigma}dW_s \text{ for } W \text{ a Brownian motion on } \reels^m \text{ under }\mathbb{P}.\\
\end{array}
\right.
\]
When the function $V$ defined in such a way is smooth we expect it to be solution of:
\begin{equation*}
\left\{
\begin{array}{c}
      \partial_t V+B(t,x,p)\cdot \nabla_x V+b(p)\cdot \nabla_p V-\sigma \Delta_p V=A(t,x,p) \text{ in } (0,T) \times \Omega\times \reels^m,\\
     V|_{t=0}=U_0.\\
\end{array}
\right. 
\end{equation*}

Which leads to the following definition of a Lipschitz solution of the master equation \eqref{MFG forward}:

\begin{definition}
Let $T>0$, $U:[0,T)\times \Omega\times\reels^m\to \reels^d$ is said to be a Lipschitz solution of \eqref{MFG forward} if :
\begin{itemize}
    \item[-] $U$ is Lipschitz in $(x,p)$ uniformly in $t\in[0,\alpha]$ for all $\alpha$ in $[0,T)$
    \item[-] for all $t<T$: \[U=\psi(t,G(\cdot,U),F(\cdot,U),U_0).\]
\end{itemize}
\end{definition}

Let us now state the following result on Lipschitz solutions in this context, which is a local existence theorem analogous to the Cauchy-Lipschitz theorem one could get for ordinary differential equations.
\begin{thm}
\label{lip sol b(p)}
    Let $(F,G,b,U_0)$ be Lipschitz in all the variables, then:
    \begin{enumerate}
        \item[-] There always is an existence time  $T>0$ such that we have a unique solution in the sense of the aforementioned definition on $[0,T)$.
        \item[-] There exist $T_c>0$ and a maximal solution $U$ defined on $[0,T_c[$ such that for all Lipschitz solutions $V$ defined on $[0,T)$ we have: $T\leq T_c$ and $U|_{[0,T)}\equiv V$.
        \item[-] If $T_c<\infty$ then $\underset{t\to T_c}{\lim} \|U(t,\cdot)\|_{Lip}=+\infty$.
    \end{enumerate}
\end{thm}
We do not prove this Theorem as it is a straightforward extension of previously existing results on Lipschitz solutions (see \cite{bertucci2023lipschitz} Theorem 1.7). 

\begin{remarque}
    For a function $f:\reels^k\to\reels^k$ we make use of the notation $\|f\|_{Lip}$ for its Lipschitz semi-norm:
    \[\|f\|_{Lip}=\underset{x\neq y}{\sup}\frac{|f(x)-f(y)|}{|x-y|}.\]
    For a function of several variables $f(x,p)$, we make the following misuse of notation for its Lipschitz constant in $x$ only:
    \[\|D_xf\|_{\infty}=\underset{x\neq y,p}{\sup}\frac{|f(x,p)-f(y,p)|}{|x-y|},\]
    as they are equal for $C^1$ Lipschitz functions.
\end{remarque}

\quad
Note that the previous result holds even if Hypothesis \ref{hyp:monotone} is not satisfied. The next Section is devoted to generalizing the a priori estimate of Proposition \ref{prop monotonicity} to functions which are only Lipschitz solutions of the equation. Namely we want to show that since we have an a priori estimate on $\|D_xU\|_\infty$, the maximal time $T_c$ given in Theorem \ref{lip sol b(p)} is equal to $+\infty$. We end the present Section by explaining how, using the equation, we can deduce a bound on $D_p U$ from a bound on $D_xU$, thus making the estimation of $D_xU$ the crucial step of the argument.

\begin{prop}
\label{b(p) U lipschitz in x implies in p}
    Assume (F,G,b,$U_0$) are Lipschitz in all the variables. For any $T>0$, there exists $C >0$ such that for any $U$ Lipschitz solution of \eqref{MFG forward} on $[0,T)$:

    \[\forall t<T, \quad \underset{s\in[0,t]}{\sup}\|D_pU(s,\cdot)\|_\infty\leq C\left(1+\exp\left(t\underset{s\in[0,t]}{\sup}\|D_xU(s,\cdot)\|_\infty\right)\right).\]
\end{prop}
\begin{proof}
    Let us fix $T>0$ and consider a Lipschitz solution $U$ on $[0,T)$.

    First, we can remark that, since $b$ depends only on $p$, $(p_s)_{s\geq 0}$ satisfies an autonomous stochastic differential equation independent of $U$. As $b$ is Lipschitz, the classical results for the continuity of SDE with respect to initial conditions are available. \\
    let $t<T$, we define:
    \[\forall u\in[0,t] \quad X^{p_0}_u=x-\int_0^uF(X^{p_0}_s,p_s,U(t-s,X^{p_0}_s,p^{p_0}_s))ds,\]
    where
    \[p^{p_0}_u=p_0-\int_0^u b(p_s)ds+\sqrt{2\sigma}W_u\]
and \[M_t=\underset{s\in[0,t]}{\sup}\|D_xU(s,\cdot)\|_\infty,\] which is finite as $t<T$.
    \begin{align*}
        \esp{|X_u^{p_0}-X_u^{q_0}|}\leq \|F\|_{Lip}\int_0^u\left( (1+M_t)\esp{|X_s^{p_0}-X_s^{q_0}|}+(1+\|D_pU(t-s,\cdot)\|_\infty)\esp{|p_s^{p_0}-p_s^{q_0}|}\right)ds  .
    \end{align*}
    From Gronwall's Lemma applied to standard Lipschitz SDE we know there exists a constant $C_b>0$ depending only on $\|b\|_{Lip}$ and $T$ such that
    \[\forall s\in[0,T],\quad \esp{|p_s^{p_0}-p_s^{q_0}|}\leq C_b|p_0-q_0|.\]
    We can conclude by yet another Gronwall's Lemma that
    \[\forall u\in[0,t],\quad  \esp{|X_u^{p_0}-X_u^{q_0}|}\leq C_be^{(1+M_t)u}|p_0-q_0|(u+\int_0^u\|D_pU(t-s,\cdot)\|_\infty ds).\]

Using the fact that $U$ is a Lipschitz solution of the master equation on $[0,t]$ gives
    \begin{align*}
        |U(t,x,p_0)-U(t,x,q_0)|&\leq \esp{\|G\|_{Lip}\int_0^t((1+M_t)|X_s^{p_0}-X_s^{q_0}|+(1+\|D_pU(t-s)\|_\infty)|p_s^{p_0}-p_s^{q_0}|)ds}\\
        &+\|U_0\|_{Lip}\esp{|X_t^{p_0}-X_t^{q_0}|+|p_t^{p_0}-p_t^{q_0}|}.
    \end{align*}
Hence
    \begin{align*}
    \frac{|U(t,x,p_0)-U(t,x,q_0)|}{|p_0-q_0|}\leq C(1+e^{M_t t}+\int_0^t \|D_pU(t-s,\cdot)\|_\infty ds),
    \end{align*}
    where $C$ is a positive constant depending only on $\|(F,G,b,U_0)\|_{Lip}$ and $T$. 
    Since
    \[\underset{p,q}{\sup} \frac{|U(t,x,p)-U(t,x,q)|}{|p-q|}=\|D_pU(t,\cdot)\|_\infty,\]
    we conclude by Gronwall's lemma that 
    \[ \underset{s\in[0,t]}{\sup}\|D_pU(s,\cdot)\|_\infty\leq C\left(1+e^{tM_t}\right).\]
As this holds true for any $t<T$, the proof is complete. 
\end{proof}

\begin{remarque}
    A similar estimate on $\|D_pU\|_\infty$  can still be obtained when b is assumed to a Lipschitz function of $(x,p)$. The proof is slightly more technical but follows from the exact same argument. Moreover, under the assumption $\sigma>0$, we may be able to get a better bound on the gradient of $U$ in $p$. Indeed, we could get estimates on $D_pU$ once $D_xU$ is bounded by freezing $x$ and considering the semi-linear parabolic system in $p$ only:
\[\partial_t U+b(p)\cdot\nabla_p U-\sigma\Delta_p U=\underbrace{G(x,p,U)-F(x,p,U)\cdot\nabla_xU.}_{\text{source term Lipschitz in }U\text{, uniformly in }x }\]
However because our estimate on $\|D_xU\|_\infty$ is only valid when b depends on p only, we would rather not rely on parabolic regularity. This allows us to consider the fully degenerate case $\sigma=0$.
\end{remarque}

\subsection{Existence of global solutions}

In this section we give sufficient conditions for the existence of a Lipschitz solution of the master equation \eqref{MFG forward} independently of the time horizon $T$. Namely, we show how we can make rigorous the a priori estimate of Proposition \ref{prop monotonicity} for Lipschitz solutions of \eqref{MFG forward}. From Proposition \ref{b(p) U lipschitz in x implies in p}, we know that this would implies a bound on $D_pU$. Hence, since there is always uniqueness of a maximal Lipschitz solution, existence and uniqueness of a global Lipschitz solution shall follow.
The proof of Proposition \ref{prop monotonicity} is based on a derivation of the system \eqref{MFG forward}. Obviously this does not generalise well to Lipschitz solutions, a first step to adapt the argument of Proposition \ref{prop monotonicity} is to state Hypothesis \ref{hyp:monotone} directly at the level of the functions $(U_0,F,G)$ rather than on their gradient. We start with a Lemma,
\begin{restatable}{lemma}{equivalence}
\label{equivalence alpha monotony}
Let $f:\reels^d \to \reels^d$ be a function of class $C^1$, the following two proposition are equivalent:

\begin{enumerate}
    \item[-] $\exists \alpha,\forall x \quad \xi\cdot Df(x)\cdot \xi\geq \alpha |Df\cdot\xi|^2. \quad \quad \quad\quad \quad \quad\quad \quad \quad(i)$
    \item[-] $\exists \alpha, \forall (x,y) \quad \langle f(x)-f(y),x-y\rangle \geq \alpha |f(x)-f(y)|^2. \quad \quad (ii)$
\end{enumerate}
\end{restatable}

\begin{remarque}
    This elementary result is already well-known, we included a proof in Appendix for the sake of completeness.
\end{remarque}
According to Lemma $\ref{equivalence alpha monotony}$, Hypothesis $\ref{hyp:monotone}$ implies that there exists $\alpha>0$ such that for all $x,y$ in $\Omega$, $u,v$ in $\reels^d$ and $p$ in $\reels^m$
 \begin{equation}
\label{U0 monotone}
\langle U_0(t,x,p)-U_0(t,y,p),x-y\rangle \geq \alpha |U_0(t,x,p)-U_0(t,y,p)|^2,
\end{equation}
\begin{align}
\label{(F,G) monotone}
\langle G(x,p,u)-G(y,p,v),x-y\rangle+&\langle F(x,p,u)-F(x,p,v),u-v\rangle \geq \alpha |x-y|^2 .
\end{align}   
The main idea of this Section is that, to overcome the lack of regularity of a Lipschitz solution, we need to adapt the proof of Proposition \ref{prop monotonicity} for another auxiliary function. Following the equivalence of Lemma \ref{equivalence alpha monotony}, we define the following:
\begin{equation}
\label{def Z b(p)}
Z(t,x,y,p)=\langle U(t,x,p)-U(t,y,p),x-y\rangle-\beta(t)|U(t,x,p)-U(t,y,p)|^2,
\end{equation}
for a Lipschitz solution $U$ of \eqref{MFG forward} and a smooth function of time $\beta$ to be chosen later on. This function shall play the role of $Z$ in the proof of Proposition \ref{prop monotonicity}. If $U$ was smooth we would expect this new function $Z$ defined in \eqref{def Z b(p)} to be a solution of
\begin{align}
\label{eq Z b(p)}
\nonumber&\partial_t Z+F^x\cdot\nabla_x Z+F^y\cdot \nabla_y Z +b(p)\cdot \nabla_p Z-\sigma \Delta_p Z\\
&=\langle G^x-G^y,x-y\rangle+\langle F^x-F^y,U^x-U^y\rangle\\
\nonumber&+2\beta(t)\sigma |\nabla_p U^x-\nabla_p U^y|^2-2\beta\langle G^x-G^y,U^x-U^y\rangle-\frac{d\beta}{dt}|U^x-U^y|^2 ,
\end{align}
where we have used the notation $U^x=U(t,x,p)$,$F^x=F(x,p,U^x)$ and $G^x=G(x,p,U^x)$. Recall that in Proposition \ref{prop monotonicity}, we wanted to apply a comparison principle on $Z$. We keep the same strategy here. We are going to show that $Z$ is a super-solution of

\begin{align}
\label{Z supersolution b(p)}
&\partial_t W+F^x\cdot\nabla_x W+F^y\cdot \nabla_y W +b(p)\cdot \nabla_p W-\sigma \Delta_p W\\
\nonumber&\geq \langle G^x-G^y,x-y\rangle+\langle F^x-F^y,U^x-U^y\rangle-2\beta\langle G^x-G^y,U^x-U^y\rangle-\frac{d\beta}{dt}|U^x-U^y|^2.
\end{align}
We now show that $Z$ satisfies equation \eqref{Z supersolution b(p)} in the viscosity sense. Let us remind the definition of a viscosity solution. Let $\mathcal{O}$ be an open subset of $\reels^k$ for some $k\in\mathbb{N}$, $\mathcal{S}_k(\reels)$ be the set of symmetric $k\times k$ matrix and $\mathcal{H}$ be a continuous mapping from $\reels^+\times\mathcal{O}\times\reels\times\reels^k\times\mathcal{S}_k(\reels)\to\reels$. Further assume that $\mathcal{H}$ is degenerate elliptic:
\begin{gather*}
\forall (t,x,u,p,A,B)\in \reels^+\times\mathcal{O}\times \reels\times \reels^k\times(\mathcal{S}_k(\reels))^2\\
A\leq B \implies \mathcal{H}(t,x,u,p,B)\leq \mathcal{H}(t,x,u,p,A),
\end{gather*}
where the inequality between $A$ and $B$ holds for Loewner order. Consider the following non linear partial differential equation
\begin{equation}
\label{visc}
\tag{E}
\partial_t u(t,x)+\mathcal{H}(t,x,u(t,x),\nabla_x u(t,x),D^2_x u(t,x))=0 \text{ for } (t,x) \text{ in } (0,T)\times\mathcal{O}\end{equation}
\begin{definition}
Let $u:(0,T)\times\mathcal{O}\to \reels$ be a continuous function.\\
(i) We say that u is a viscosity supersolution of \eqref{visc} if for any $\varphi\in C^{1,2}((0,T)\times \mathcal{O})$ such that a point of minimum $(t^*,x^*)$ of $u-\varphi$ is achieved in $(0,T)\times\mathcal{O}$ the following holds
\[\partial_t \varphi(t^*,x^*)+\mathcal{H}(t^*,x^*,u(t^*,x^*),\nabla_x \varphi(t^*,x^*),D^2_x\varphi(t^*,x^*))\geq 0.\]
(ii) We say that u is a viscosity subsolution of \eqref{visc} if for any $\varphi\in C^{1,2}((0,T)\times \mathcal{O})$ such that a point of maximum $(t^*,x^*)$ of $u-\varphi$ is achieved in $(0,T)\times\mathcal{O}$ the following holds
\[\partial_t \varphi(t^*,x^*)+\mathcal{H}(t^*,x^*,u(t^*,x^*),\nabla_x \varphi(t^*,x^*),D^2_x\varphi(t^*,x^*))\leq 0.\]
(iii) We say that $u$ is a viscosity solution of \eqref{visc} if $u$ is both a viscosity supersolution and a viscosity subsolution of \eqref{visc}.
\end{definition}
Note that in all generality, the definition of a viscosity solution can be extended beyond the continuous setting to function $u$ that are only locally bounded. One of the key features of viscosity solutions is that under some structural assumptions on the operator $\mathcal{H}$ they satisfy a comparison principle. For more details on viscosity solutions we refer to \cite{crandall1992users} for a general introduction and to \cite{touzi} for an introduction to viscosity solutions in the context of optimal control. 
\quad

\begin{lemma}
\label{lemma Z supersol b(p)}
    Assume that $(F,G,b,U_0)$ are Lipschitz. Let U be a Lipschitz solution of the master equation \eqref{MFG forward} on $[0,T)$ for $T>0$ and Z be defined as in \eqref{def Z b(p)}.\\
    Then Z is a viscosity supersolution of \eqref{Z supersolution b(p)} on $(0,T)$
\end{lemma}
\begin{proof}
\quad

\noindent\textit{A relation satisfied by $Z$}\\
To prove this lemma, we make extensive use of the representation formula for Lipschitz solutions. For $t< T$ $U$ is given by 
\[\left\{
\begin{array}{c}
U(t,x,p)=\displaystyle\mathbb{E}\left[\left.\int_0^{t}G(X_s,p_s,U(t-s,X_s,p_s))ds+U_0(X_t,p_t)\right|X_0=x,p_0=p\right],\\
dX_s=-F(X_s,p_s,U(t-s,X_s,p_s))dt,\\
dp_s=-b(p_s)dt+\sqrt{2\sigma}dW_s.\\
\end{array}
\right. 
\]
To simplify computations we make use of the notation 
\[G_s^{x,p}=G(X_s,p_s,U(t-s,X_s,p_s)),\]
for the process $(X_t,p_t)_{t\geq 0}$ starting in $(x,p)$. Let us first remark that $U$ satisfies a dynamic programming principle
    \begin{equation}
    \label{dpp b(p)}
    U(t,x,p_0)=\esp{U(s,X^x_{t-s},p^{p_0}_{t-s})+\int_0^{t-s} G(X^x_{u},p^{p_0}_{u},U(t-u,X^x_{u},p^{p_0}_{u}))du}.
    \end{equation}
    This follows from the definition of a Lipschitz solution. Let us introduce \[W(t,x,y,p)=\langle U(t,x,p)-U(t,y,p),x-y\rangle.\] From the previous dynamic programming principle, the following holds
\begin{align}
\label{Z beta=0 b(p)}
\nonumber &\quad \quad \quad \quad \quad \quad \quad \quad\quad \quad \quad \quad \quad \quad \quad \quad W(t,x,y,p_0)=\\
&\esp{W(s,X^x_{t-s}Y^y_{t-s},p^{p_0}_{t-s})+\int_0^{t-s}\left(\langle G_u^{x,p_0}-G_u^{y,p_0},x-y\rangle+\langle F_u^{x,p_0}-F_u^{y,p_0},U_s^{x,p_0}-U_s^{y,p_0}\rangle\right)du}.
\end{align}
Let us also remark that by using the dynamic programming principle \eqref{dpp b(p)} and Jensen inequality:
\begin{equation}
\label{martingale inequality b(p)}
|U(t,x,p_0)-U(t,y,p_0)|^2\leq \esp{\left|U_s^{x,p_0}-U_s^{y,p_0}+\int_0^{t-s}(G_u^{x,p_0}-G_u^{y,p_0})du\right|^2}.
\end{equation}
We are now ready to prove the statement.
\begin{align*}
    Z(t,x,y,p)&=W(t,x,y,p)-\beta(t)|U(t,x,p)-U(t,y,p)|^2\\
    &=\underbrace{W(t,x,y,p)}_{I}-\beta(s)\underbrace{|U(t,x,p)-U(t,y,p)|^2}_{II}-(\beta(t)-\beta(s))|U(t,x,p)-U(t,y,p)|^2 .
\end{align*}
Applying \eqref{Z beta=0 b(p)} to $I$ and \eqref{martingale inequality b(p)} to $II$, and developing the square in inequality \eqref{martingale inequality b(p)} implies that
\begin{align}
\label{inequalityZ}
    \nonumber Z(t,x,y,p_0)&\geq \esp{Z(s,X^{x}_{t-s},Y^y_{t-s},p^{p_0}_{t-s})-\beta(s)\|\int_0^{t-s}(G_u^{x,p_0}-G_u^{y,p_0})du\|^2}\\
    &+\esp{\int_0^{t-s}\left(\langle G_u^{x,p_0}-G_u^{y,p_0},x-y\rangle+\langle F_u^{x,p_0}-F_u^{y,p_0},U_s^{x,p_0}-U_s^{y,p_0}\rangle\right)du}\\
    \nonumber &-(\beta(t)-\beta(s))\|U(t,x,p_0)-U(t,y,p_0)\|^2-2\beta(s)\esp{\int_0^{t-s}\langle U_s^{x,p_0}-U_s^{y,p_0},G_u^{x,p_0}-G_u^{y,p_0}\rangle|}.
\end{align}
 Moreover, remark that such an inequality would still holds if instead of a deterministic time s we were to consider an almost surely bounded stopping time $\tau$ in $[0,t]$.

 \quad
 
 \noindent\textit{Z is a viscosity supersolution}\\
Let us now assume that for some $(t,x,y,p_0)\in(0,T)\times\Omega^2\times\reels^m$ there exists a smooth function $\varphi$ such that 
\[Z(t,x,y,p_0)-\varphi(t,x,y,p_0)=\min(Z-\varphi)=0,\]
and define the stopping time 
\[\tau_h=\inf\{s>0 \quad (s,p_{s})\notin [0,h]\times B(p_0,1)\}.\]
Where $B(p_0,1)=\{q\in\reels^m \quad \|q-p_0\|\leq 1\}$.
Because $Z(t-\tau_h,X^x_{\tau_h},Y^y_{\tau_h},p^{p_0}_{\tau_h})\geq \varphi(t-\tau_h,X^x_{\tau_h},Y^y_{\tau_h},p^{p_0}_{\tau_h})$, holds almost surely and it becomes an equality at $(t,x,y,p_0)$ it is easy to see that $\varphi$ satisfies inequality \eqref{inequalityZ} with the stopping time for any $h\leq 1$. We may then apply Ito's lemma to $\varphi$ and divide by $h$. Since $\varphi,D\varphi,D^2\varphi,F,G,b,U$ are all continuous function on $[t-1,t]\times \Omega\times B(p_0,1)$ which is compact, they are bounded on this set. By the dominated convergence theorem, we finally arrive at
\begin{align*}
&\partial_t \varphi+F^x\cdot\nabla_x \varphi+F^y\cdot \nabla_y \varphi +b(p)\cdot \nabla_p \varphi-\sigma \Delta_p \varphi\\
&\geq \langle G^x-G^y,x-y\rangle+\langle F^x-F^y,U^x-U^y\rangle-2\beta\langle G^x-G^y,U^x-U^y\rangle-\frac{d\beta}{dt}|U^x-U^y|^2 .
\end{align*}
Because this holds at all points of minimum $(t,x,y,p_0)\in (0,T)\times \Omega^2\times \reels^m$ of $Z-\varphi$ for any smooth function $\varphi$,  $Z$ is indeed a viscosity supersolution of \eqref{Z supersolution b(p)}.
\end{proof}

This slightly technical Lemma is the cornerstone that will allows us to show that monotonicity estimates are still valid for Lipschitz solution. We can now state the main result of this Section:
\begin{thm}
\label{theorem: existence b(p)}
   Assume that $(U_0,G,F,b)$ are Lipschitz in all variables, under Hypothesis \ref{hyp:monotone} there exists a unique global Lipschitz solution of \eqref{MFG forward}.
\end{thm}
\begin{proof}
Because we already got an estimate on $\|D_pU\|_\infty$ in terms of $\|D_xU\|_\infty$ we need only to show that $\|D_xU\|_\infty$ does not blow up in finite time. All functions $(U_0,G,F,b)$ are Lipschitz, hence there exists a time $T>0$ and an associated Lipschitz solution $U$ to \eqref{MFG forward} on $[0,T)$. Let $Z$ be defined as in \eqref{def Z b(p)}, we know thanks to Lemma \ref{lemma Z supersol b(p)} that on any time interval $[0,\tau],\quad \tau<T$, $Z$ is a viscosity supersolution of \eqref{Z supersolution b(p)}. Now, using assumption \eqref{(F,G) monotone} on the monotonicity of $(G,F)$, we get that for any $(t,x,y,p)$
\begin{align*}
&\langle G^x-G^y,x-y\rangle+\langle F^x-F^y,U^x-U^y\rangle-2\beta\langle G^x-G^y,U^x-U^y\rangle-\frac{d\beta}{dt}|U^x-U^y|^2\\
&\geq (\alpha-\|D_xG\|_\infty^2\beta) |x-y|^2-(\beta \|D_uG\|_\infty+\frac{d\beta}{dt})|U^x-U^y|^2.
\end{align*}
 Fixing $\beta_0\leq \alpha$ and $\beta(t)=\beta_0e^{-\lambda t}$ with $\lambda\geq \|D_uG\|_\infty$, we finally get that $Z$ is a viscosity supersolution of 
 \begin{equation}
 \label{Z gamma supersol b(p)}
\partial_t Z+F^x\cdot\nabla_x Z+F^y\cdot \nabla_y Z +b(p)\cdot \nabla_p Z-\sigma \Delta_p Z\geq 0.
 \end{equation}
Since $\beta_0$ was choosen smaller than $\alpha$, our assumption on $U_0$ also ensure that
\[Z_{t=0}\geq 0.\]
We may now conclude by a comparison principle that Z stay positive over the course of time for any time interval $[0,\tau]$ with $\tau<T$. We can apply Theorem 2.5 of \cite{34bed8a7-6223-3f55-b20d-c7daa21c8c15} on such time interval which implies:
\[\underset{[0,\tau]}{\inf} Z\geq 0.\]
It only remains to show that the non-negativity of $Z$ implies an estimate on its Lipschitz bound in $x$. Take $t<T, (x,y,p) \in \Omega^2\times \R^m$, from the non-negativity of $Z$
\begin{align*}
     |U(t,x,p)-U(t,y,p)\|x-y|\geq \beta(t) |U(t,x,p)-U(t,y,p)|^2.
\end{align*}
We then deduce \[\frac{1}{\beta(t)}\geq \frac{|U(t,x,p)-U(t,y,p)|}{|x-y|}.\]
Taking the supremum over $(x,y,p)$, we finally get
\[\forall t<T, \quad \|D_xU(t,\cdot)\|_\infty \leq \frac{1}{\beta(t)}.\]
Since we have an uniform bound on the Lipschitz norm of $U$ in $x$, it is then easy to see that the blow up time $T$ of our Lipschitz solution $U$ must satisfies $T=+\infty$ to avoid any contradiction, hence the existence of a global Lipschitz solution of the problem.  
 \end{proof}

\section{Noise depending on the state of the game}\label{sec:non-autonomous}
In this Section we investigate the more intricate case of the master equation 
\begin{equation}
\label{MFG b(x,p,u)}
\left\{
\begin{array}{c}
\partial_t U+F(x,p,U)\cdot \nabla_x U+b(x,p,U)\cdot \nabla_p U-\sigma \Delta_p U=G(x,p,U) \text{ for }  t\in(0,T), x \in \Omega,p \in\R^m, \\
U(0,x,p)=U_0(x,p) \text{ for } (x,p)\in \Omega\times \reels^m.
\end{array}
\right. 
\end{equation}
The only difference with \eqref{MFG forward} is that now we allow $b$, the drift of the stochastic process $p$, to depend both on $x$ and the value $U$ itself. We keep up with the assumption we made on $\Omega \subset \R^d$.
\subsection{Preliminary remarks}
\subsubsection{Modeling}

In our opinion, the fact that $b$ can depend on $(x,U)$ is natural in a wide variety of applications. Indeed, consider for instance the case in which $p$ models the price of a stock. In such a situation, it seems fair for its evolution to depend on the distribution of players $x$ and their decisions (captured by $U$ itself in this finite states setting). In general, common noise models exogenous or environmental factors. We believe realistic to assume that in certain contexts, the players have an effect on their environment! 

Let us insist that at the moment, there seems to have been remarkably few interest in the systematic study of PDE arising from this kind of model. In fact to the extent of our knowledge the only preexisting results were given on Forward backward stochastic differential equations (FBSDE) rather than on the associated partial differential equation \cite{doi:10.1137/S0363012996313549}. 

\subsubsection{Necessity of new estimates}
In this setting, the definition of what is a Lipschitz solution has to be slightly changed to account for non linearity introduced by the dependency of $b$ in $U$. Indeed we now define $\tilde{\psi}$ for three vector fields $(A,B^1,B^2)$:
\[\left\{
\begin{array}{c}
 \displaystyle\tilde{\psi}(T,A,B^1,B^2,U_0)=V,\\
 \text{ where } \forall t\in[0,T), (x,p)\in\Omega\times\reels^m \quad V(t,x,p)=\displaystyle\mathbb{E}\left[\left.\int_0^{t}A(t-s,X_s,p_s)ds+U_0(X_t,p_t)\right|X_0=x,p_0=p\right],\\
dX_s=-B^1(t-s,X_s,p_s)dt,\\
dp_s=-B^2(t-s,X_s,p_s)dt+\sqrt{2\sigma}dW_s.\\
\end{array}
\right.
\]
and we define a Lipschitz solution as a fixed point of this new functional. 
\begin{definition}
Let $T>0$, $U:[0,T)\times \reels^d\times\reels^m\to \reels^d$ is said to be a Lipschitz solution of \eqref{MFG b(x,p,u)} if :
\begin{itemize}
    \item[-] U is Lipschitz in $(x,p)$ uniformly in $t\in[0,\alpha]$ for all $\alpha$ in $[0,T)$
    \item[-] for all $t<T$: \[U=\tilde{\psi}(t,G(\cdot,U),F(\cdot,U),b(\cdot,U),U_0)\]
\end{itemize}
\end{definition}
Previous results of local existence for Lipschitz solutions still hold in this more involved case (such as in \cite{bertucci2023lipschitz}, Theorem 2.4). We want to stress that global gradient estimates obtained in the autonomous case do not hold in general in this new model. While we expect that we could still get estimates on the gradient in $p$ in function of the gradient in $x$ using parabolic regularity, it is not so clear that $D_xU$ does not blow up in finite time.

\paragraph{Difficulties arising from the dependency of b on U}
\quad

Before giving an example of an explosion of $D_xU$, we state a Lemma on solutions of the master equation which are invertible in $x$. This is a natural extension of a result presented by P.L. Lions during his lectures \cite{P_L_Lions} to the master equation we are studying. 
\begin{lemma}
\label{invertible solution}
    Let $U$ be a regular solution of 

    \begin{equation*}
\left\{
\begin{array}{c}
     \partial_t U+F(x,p,U)\cdot \nabla_x U+b(x,p,U)\cdot \nabla_p U-\sigma \Delta_p  U=G(x,p,U) \quad (t,x,p)\in (0,T)\times\reels^d\times\reels^{m},\\
     U(0,x,p)=U_0(x,p).\\
\end{array}
\right.
\end{equation*}
 Assume that for all $(t,p)$ $U$ is invertible in $x$ and that $V$, its inverse in $x$, is smooth in all variables. Then $V$ is a solution of
    \begin{equation*}
\left\{
\begin{array}{c}
     \partial_t V+G(V,p,y)\cdot \nabla_y V+b(V,p,y)\cdot \nabla_p V-\sigma \Delta_p V=F(V,p,y) \quad (t,y,p)\in (0,T)\times\reels^d\times\reels^{m},\\
     V(0,y,p)=U_0^{-1}(y,p).\\
\end{array}
\right.
\end{equation*}
with the misuse of the notation $U_0^{-1}(y,p)=\left(U_0(\cdot,p)\right)^{-1}(y)$.
\end{lemma}
\begin{proof}
Let $1\leq i \leq d$ and $t<T$, we know that:
\[V^i(t,U(t,x,p),p)=x^i.\]
As a consequence:
\begin{align}
\label{deriv p inverse}
    \nonumber 0&=\frac{d}{d p_k}V^i(t,U(t,x,p),p),\\
    0&=\partial_{p_k} V^i+\partial_{y_j}V^i\partial_{p_k}U^j,
\end{align}
where we used Einstein notation for summation. Similarly,
\begin{align*}
0&=\frac{d}{dt}V^i(t,U(t,x,p),p),\\
&=\partial_t V^i+\partial_{y_j} V^i \partial_t U^j.
\end{align*}
Plugging the equation satisfied by $U$ in this last equality leads to
\begin{align*}
0=\partial_t V^i+\partial_{y_j} V^i \left(G^j-F^k\partial_{x_k}U^j-b^k\partial_{p_k}U^j+\sigma\sum_k \partial^2_{p_k}U^j.\right).
\end{align*}
Defining $y=U(t,x,p)$ and $x=V(t,y,p)$ gives
\[\partial_{y_j}V^i(t,U(t,x,p),p)G^j(x,p,U(t,x,p))=G(V,p,y)\cdot \nabla_y V^i.\]
Because $\partial_{y_j}V^i \partial_{x_k} U^j=\delta_{ik}$ for the Kronecker delta symbol,
\[F^k\partial_{y_j}V^i \partial_{x_k} U^j=F^i(x,p,U(t,x,p))=F^i(V(t,y,p),p,y).\]
For the last term involving $\sigma$, 
\begin{align*}
 \sigma \sum_k \partial^2_{p_k}U^j \partial_{y_j} V^i&=\sigma\sum_k \partial_{p_k}\left(\partial_{p_k}U^j\partial_{y_j} V^i\right)-\partial_{p_k}U^j\partial^2_{p_k,y_j}V^i-\partial_{p_k}U^j\partial^2_{y_j,y_l} V^i \partial_{p_k} U^l\\
 &=\sigma\sum_k \partial_{p_k}\left(\partial_{p_k}U^j\partial_{y_j} V^i\right)-\partial_{p_k}U^j\partial_{y_j}\left(\partial_{p_k}V^i+\partial_{y_l} V^i \partial_{p_k} U^l\right).
\end{align*}
Making use of \eqref{deriv p inverse} we can write
\[\sigma \sum_k \partial^2_{p_k}U^j \partial_{y_j} V^i=-\sigma\sum_k \partial^2_{p_k} V^i.\]
Finally, we get
\[\partial_t V^i+G(V,p,y)\cdot \nabla_y V^i-F^i+b(V,p,y)\cdot\nabla_p V^i-\sigma \Delta_p V^i=0.\]
\end{proof}
\begin{remarque}
While we decided to state this result on $\reels^d$, in the proof we only use the fact that $U$ is a one to one function. Taking an equation on the whole space is a choice of convenience as otherwise the domain of $y\to V(t,y,p)$ depends on $(t,p)$.
\end{remarque}
The previous Lemma yields an analogy between the dependence of $b$ on $x$ and on $U$. Before giving an example of explosion in the case in which $b$ depends on $x$, we present an example in which we highlight this formal change of variable. 
\begin{exemple}
\label{b(U) inversible}
Consider the equation\begin{align*}
    \begin{array}{c}
      \partial_t U+F(U)\cdot \nabla_x U+b(U)\cdot \nabla_p U-\sigma \Delta_p U = 0 \quad (t,x,p)\in (0,T)\times\reels^d\times\reels^{m}, \\
      U(0,x,p)=U_0(x,p) \quad x\in \reels^d, p\in \reels^{m} .
    \end{array}
\end{align*}
For $U_0$ invertible in $x$, we note $V$ the inverse of $U$ in $x$. Following Lemma \ref{invertible solution} we expect V to be a solution of:
\begin{align*}
    \begin{array}{c}
      \partial_t V+b(y)\cdot \nabla_p U-\sigma \Delta_p U = F(y) \quad (t,y,p)\in (0,T)\times\reels^d\times\reels^{m} \\
      V(0,y,p)=U^{-1}_0(y,p) \quad y\in\reels^d, p\in \reels^{m},
    \end{array}
\end{align*}
when it is well defined. From which we can deduce the expression of V by integrating along the characteristics
\[V(t,y,p)=\displaystyle\mathbb{E}\left[tF(y)+V_0(y,p-tb(y)+\sqrt{2\sigma}W_t)\right].\]
It then follows that
\begin{align*}
 D_yV(t,y,p)&=\displaystyle\mathbb{E}\left[tD_uF(y)+D_yV_0(y,p-tb(y)+\sqrt{2\sigma}W_t)-tD_pV_0(y,p-tb(y)+\sqrt{2\sigma}W_t) D_u b\right]\\
&=\displaystyle\mathbb{E}\left[(D_xU_0)^{-1}\left(I_d+t(D_xU_0 D_uF(y)+D_pU_0 D_u b)\right)(y,p-tb(y)+\sqrt{2\sigma}W_t)\right],
\end{align*}
where the last line was obtained using the fact that $U$ and $V$ are inverse of each others,
\[D_p (U\circ V)=D_pU(V)+D_xU(V)D_pV=0.\]
If there exists a point $(t^*,y^*,p^*)$ such that $D_yV(t^*,y^*,p^*)$ is not invertible, then we expect that 
\[\underset{(t,y,p)\to(t^*,y^*;p^*)}{\lim} \|D_xU(t,V(t,y,p),p)\|\to +\infty.\]
In the case $\sigma = 0$, a sufficient condition to avoid such an explosion could be
\[\forall (x,p)\in \reels^d\times \reels^m \quad D_xU_0(x,p)D_uF(U_0(x,p))+D_pU_0(x,p)D_ub(U_0(x,p))\geq 0.\]
This clearly imposes structural conditions on $b$ and $U_0$.
\end{exemple}

\paragraph{Difficulties arising from the dependency of b on x}
It is quite natural that further assumptions and restrictions be required when b depends on $U$. However, the dependence of $b$ in $x$, even if it does not add any additional non-linearity in the problem, can lead to the apparition of singularities in finite time. We now give such an example.
\begin{exemple}
\label{b(x) lineaire}
Let us a consider a toy model in dimension one with 
\[
\left\{
\begin{array}{c}
     F(x,p,U)=U,  \\
     G(x,p,U)=x,\\
     b(x,p)=\lambda x,\\
     U_0(x,p)=\alpha_0 x+\beta_0 p.
\end{array}
\right.
\]
Looking for a linear solution $U(t,x,p)=\alpha(t) x+\beta(t) p$ of 
\[ \partial_t U+U\partial_x U+\lambda x \partial_p U-\sigma \partial^2_p U=x \text{ in } (0,\infty)\times \R^2,\]
we get
\[
\left\{
\begin{array}{c}
    \frac{d\alpha}{dt}(t)+\alpha(t)^2+\lambda \beta(t)=1,  \\
    \frac{d\beta}{dt}(t)+\alpha(t)\beta(t)=0,\\
    \alpha(0)=\alpha_0,\\
    \beta(0)=\beta_0.
\end{array}
\right
.
\]
We may easily solve for $\beta$ assuming $\alpha$ is well behaved
\[\beta(t)=\beta_0e^{-\int_0^t \alpha(s)ds}.\]
Hence $\alpha$ is solution to
\[
\left\{
\begin{array}{c}
    \frac{d\alpha}{dt}(t)+\alpha(t)^2+\lambda \beta_0e^{-\int_0^t \alpha(s)ds}=1,  \\
    \alpha(0)=\alpha_0.
\end{array}
\right.
\]
Let us now take $\lambda \beta_0\geq 1$.
Assume that $\alpha_0<0$ at time 0, then
\[\frac{d\alpha}{dt}(0)\leq -\alpha^2(0)<0.\]
As a consequence $\alpha$ must be decreasing in a neighbourhood of zero and we deduce that on this set
\[\frac{d\alpha}{dt}(t)\leq -\alpha^2(t)<0.\]
Because this leads to $\alpha$ further decreasing, we deduce that the inequality must be true for all times and hence that there exist $T<+\infty$ such that  $\underset{t\to T}{\lim } \alpha(t)=-\infty$. This was to be expected as in this very simple case the monotonicity condition on $U_0$ reduces to $\alpha_0\geq 0$. 

Let us now assume that $\alpha_0>0$ so that $U_0$ is strongly monotone in $x$, it only remains to show that for a well chosen pair $(\lambda,\beta_0)$ there exists a time t such that $\alpha(t)<0$, using the previous result this will then be enough to ensure $\alpha$ blows up in finite time. Take a couple $(\lambda,\beta_0)$ satisfying
\[1-\lambda\beta_0e^{-\alpha_0 (\alpha_0+1)}\leq -1. \]
For this choice, $\alpha$ must be decreasing around 0 and by a similar argument as the one used earlier we deduce that
\[\forall t\in [0,\alpha_0+1],\frac{d\alpha}{dt}(t)\leq -1.\]
Hence $\alpha$ becomes negative in finite time.

Since $\partial_x U(t,x,p)=\alpha(t)$ this is indeed a blow up of the gradient of $U$ in finite time. This shows that despite the regularizing effects from the (strong in this case) monotonicity\footnote{Indeed, it was shown in \cite{bertucci2021master} that in the absence of $b$, such an equation regularizes the initial condition.} of both $F$ and $G$, the dependency of $b$ in $x$ might be enough to perturb the well-posedness of the equation, even though it did not add further non linearity. 
Let us also remark the regularizing effect the monotonicity of $b$ in $p$ would have on the equation. In this linear setting this translates as
\[b(x,p)=\lambda x+\mu p \quad \mu>0.\]
Hence $\beta$ would become:
\[\beta(t)=\beta_0e^{-\mu t-\int_0^t \alpha(s)ds}.\]
Adding the term in $e^{-\mu t}$ tends to make $\beta(t)$ smaller in absolute value (though it might not be enough to compete with the term in $\alpha$).
\end{exemple}
\subsection{A heuristic on the mean field game master equation}
As we just saw, the estimates of Section 2 are in general not true anymore when $b$ depends on $(x,U)$. However under additional assumptions, we may recover an estimate stemming from a similar idea to the one that led to Proposition \ref{prop monotonicity}. As the proof is somewhat calculation intensive (although very similar to the one in the case where $b$ depends only on $p$), we present first an heuristic that shows from where both the result, and the hypothesis which might not seems very intuitive at first glance, come from.

\quad

Let us first take a look at first order Hamilton Jacobi Bellman (HJB) equations.
Let $\varphi$ be a solution of:
\begin{equation*}
\left\{
\begin{array}{c}
     \partial_t \varphi+H(x,\nabla_x \varphi)=0 \text{ in } (0,\infty)\times \R^d,\\
     \varphi|_{t=0}=\varphi_0.\\
\end{array}
\right.
\end{equation*}
If we fix $U=\nabla_x \varphi$, then we know that $U$ is solution of
\begin{equation} \partial_t U+F(x,U)\cdot \nabla_x U= G(x,U)\text { in } (0,\infty)\times \R^d,\end{equation}
for $F(x,u)=D_u H(x,u)$ and $G(x,u)=-D_xH(x,u)$.  In particular this shows that the propagation of monotonicity (and as such of regularity) on U is equivalent to the propagation of convexity on $\varphi$. In general, we don't have to consider a potential game (ie the gradient of a HJB equation) so long as the equations keeps the same structure. Let us also mention that equation \eqref{MFG forward} falls into this category of equation as we can easily see it by taking the gradient in $x$ of the equation
\begin{equation*}
\left\{
\begin{array}{c}
     \partial_t \varphi+H(x,p,\nabla_x \varphi)+b(p)\cdot\nabla_p \varphi -\sigma \Delta_p \varphi=0, \\
     \varphi|_{t=0}=\varphi_0.\\
\end{array}
\right.
\end{equation*}
This explains why estimates previously obtained without the variable $p$ still held true in this setting: the fundamental structure of the equations was not changed. However it is easy to see that such considerations do not hold anymore when $b$ depends on more than just $p$. In this case, we would see a term in $\nabla_p\varphi$ come out in the equation satisfied by $\nabla_x \varphi$ which cannot be expressed in term of this last function alone. We remark, nonetheless, that we would obtain an equation with the same structure if we were to look at the one satisfied by $(\nabla_x\varphi,\nabla_p \varphi)$. With this insight, we may try to "complete" equation \eqref{MFG b(x,p,u)} by another. Let $V$ be a solution of
\begin{equation}
\label{equation V}
\left\{
\begin{array}{c} \\
     \partial_t V+F(x,p,U)\cdot\nabla_x V+b(x,p,U)\cdot\nabla_p V-\sigma\Delta_p V=Q(x,p,U,V) \text{ in }(0,\infty)\times \reels^d\times \reels^m\\
     V|_{t=0}=V_0(x,p),
\end{array}
\right. 
\end{equation}
where $Q$ and $V_0$ are $\R^m$ valued Lipschitz functions.
Considering the equation satisfied by $\Bar{U}=(U,V)$ we get
\[\partial_t \Bar{U}+\Bar{F}(X,\Bar{U})\cdot \nabla_X \Bar{U}-\sigma \Delta_p \Bar{U}=\Bar{G}(X,\Bar{U}),\]
with $X=(x,p)$, $\Bar{F}(X,\Bar{U})=(F,b)(x,p,U)$ and $\Bar{G}(X,\Bar{U})=(G,Q)(x,p,U,V)$.\\
Following the proof of Proposition \ref{prop monotonicity} we expect that under the assumptions
\begin{equation}
\label{V0 monotone}
\langle \bar{U}_0(X)-\bar{U}_0(Y),X-Y\rangle\geq \alpha|\bar{U}_0(X)-\bar{U}_0(X)|^2
\end{equation}
\begin{align}
\label{V monotone fgb}
\nonumber \langle \Bar{G}(X,\Bar{U}(t,X))-\Bar{G}(Y,\Bar{U}(t,Y)),X-Y\rangle+&\langle \Bar{F}(X,\Bar{U}(t,X))-\Bar{G}(Y,\Bar{U}(t,Y)),\bar{U}(t,X)-\bar{U}(t,Y)\rangle\\
&\geq \alpha |X-Y|^2,
\end{align}
we will be able to recover global Lipschitz estimates from monotonicity on the couple $(U,V)$ and hence on $U$.
Let us now discuss a bit on the choice of the function $V$ to "complete" $U$. So far we have stayed rather ambiguous on whether $(V_0,Q)$ are fixed and $V$ is to be considered as an additional unknown or $V$ is a chosen smooth function and $(V_0,Q)$ are defined through \eqref{equation V}. We claim that both approach are possible. However because condition \eqref{V monotone fgb} is also a condition on $V$, it is hard to keep track of it without some a priori knowledge on $V$. This is why the second approach is much more practical. Condition \eqref{V0 monotone} implies that $V_0$ must be monotonous in $p$, the simplest choice of such a function $V$ is to take:
\[V(t,x,p)=V_0(x,p)=Ap, \]
with $A\in \mathcal{M}_m(\reels)$ a positive definite matrix. Defining $Q$ through \eqref{equation V} we deduce that 
\[Q(x,p,U)=Ab(x,p,U).\]
In light of this fact and the linearity of $V$, Conditions \eqref{V0 monotone} and \eqref{V monotone fgb} reduce to
\begin{hyp}\label{hyp:b(x,p,U)}
There exist a matrix $A\in\mathcal{M}_m(\reels)$ and $\alpha>0$ such that for all $x,y$ in $\Omega$, $p,q$ in $\reels^m$ and $u,v$ in $\reels^d$:
\begin{align}
\label{condition U_0 V lineaire }\langle U_0(x,p)-U_0(y,q),x-y\rangle+&\langle p-q,A(p-q)\rangle\geq \alpha|U_0(x,p)-U_0(y,q)|^2,\\
\nonumber\\
\nonumber\langle G(x,p,u)-G(y,q,v),x-y\rangle+\langle F(x,&p,u)-F(x,q,v),u-v\rangle+\langle b(x,p,u)-b(y,q,v),(A+A^T)(p-q)\rangle \\
\label{(G,F,Ab)}    &  \geq \alpha (|x-y|^2+|p-q|^2).
\end{align}
\end{hyp}
Because both conditions can be expressed only in terms of $A+A^T$, we may assume without loss of generality in the rest of the paper that $A$ is a symmetric matrix. The condition \eqref{(G,F,Ab)} is then simply a monotonicity condition on $(G,F,Ab)$. If there exists a matrix $A$ such that both conditions hold, we expect that we are able to get estimates on $U$. Let us insist on the fact that both conditions on $U_0$ and $A$ are not trivial. For example \eqref{condition U_0 V lineaire } implies that the gradient of $U$ in $p$ must vanish when the gradient in $x$ is itself 0. Of course, if we have a strong non-degeneracy condition on $D_xU_0$, say for instance, uniform coercivity in $(x,p)$ then we may always find a suitable $A$. Note that for potential games, this means $U_0$ is the gradient of a strongly convex function in $x$ uniformly in $p$. We can also give the following example that do not require coercivity of $D_xU_0$.
\begin{exemple}
Let $\Psi:\reels^d\to \reels^d$ be such that
\[\exists \alpha>0 \quad \forall (x,y)\in \reels^{2d} \quad \langle \Psi(x)-\Psi(y),x-y\rangle \geq \alpha |\Psi(x)-\Psi(y)|^2.\]
 and $f:\reels^m\to \reels^d$ be Lipschitz. Then there exists $A\in \mathcal{M}_m(\reels)$ such that 
 \[U_0(x,p)=\Psi(x+f(p))\] satisfies the condition \eqref{condition U_0 V lineaire }.
\end{exemple}
It does becomes much simpler though, if we assume that $U_0$ is not a function of $p$, as the condition on $(U_0,A)$ then reduces to the $\alpha-$monotonicity of $U_0$ in $x$, and because we can choose any matrix $A$ positive semi-definite, the joint monotony of $(G,F,b)$ become a sufficient condition to propagate estimates on U. While this was a fairly obvious observation, we still present it as we believe that it is natural in some applications to have $U_0(x,p)\equiv U_0(x)$ (remember that $p$ models noise). 
Nevertheless, we don't believe this monotonicity condition on $b$ to be unreasonable. Consider the following example in dimension 1

\begin{exemple}
    Let \[b(x,p)=(r_0+\frac{r(x)}{1+p})p\text{ for }(x,p)\in\Omega\times \reels^+,\]
    with $r_0>0$ and $\underset{x}{\inf} (r_0+r(x))=\alpha>0$ and assume $(F,G)$ satisfy Hypothesis \ref{hyp:monotone}.
    Then if we assume either that $\|D_pF\|_\infty$ and  $\|D_p G\|_\infty$ or $\|r\|_{Lip}$ are sufficiently small, there exist an $A\in \reels^+$ such that the condition \eqref{(G,F,Ab)} is fulfilled.
\end{exemple}
This is an interesting example as it allows us to consider SDEs of the form 
\[dp_t=p_t(r_tdt+\sigma dW_t),\]
which may be used to model market prices or other positive quantities of interest. In this example $r_0$ could model a flat rate for inflation while $r(x)$ correspond to the dependence of the asset price on the distribution of the players $x$ and the term $\frac{1}{1+p}$ accounts for inertia of the price. The condition on $r+r_0$ can then be understood as the fact that the market tends to drive up the price of the asset which is arguably a strong assumption.  In this particular case a value function can be defined for $p$ in $\reels^+$ only without additional boundary condition as $0$ is an absorbing point for the process $(p_t)$. Of course since the Brownian term depends on $p$ this model does not quite correspond to equation \eqref{MFG b(x,p,u)}. Let us state that for $\sigma$ sufficiently small, this is of no concern as we delay the analysis of this more convoluted case to section \ref{section: vol}
\color{black}
\subsection{Gradient estimate }

Following the heuristic we just described, we state and prove the following gradient estimate. As in the previous Section, we are not going to use directly this estimate as we shall work with Lipschitz solutions. We still present it so that the link with the autonomous case is made more transparent.
\begin{prop}
\label{prop:estimate}
Let $U$ be a smooth solution of the master equation \eqref{MFG b(x,p,u)} on $(0,T)$, uniformly Lipschitz continuous in $(x,p)$. Assume that $(G,F,b)$ is smooth, Lipschitz in all the variables and satisfies \eqref{(G,F,Ab)}. Further assume that $U_0$ is such that:
\begin{equation}
\label{U0 V lineaire diff}
\exists A\in\mathcal{M}_m(\reels),\alpha>0,\quad \left(\begin{array}{c}\xi_1\\ \xi_2\end{array}\right)^T\left(\begin{array}{cc} D_xU_0\quad D_p U_0\\ 0\quad A\\\end{array}\right)\left(\begin{array}{c}\xi_1\\ \xi_2\end{array}\right)\geq \alpha |D_x U_0\cdot \xi_1|^2, \end{equation}
Then U is monotone in x on $[0,T)$ and there exists constants $C_x,C_p$ depending on $G,F,b,U_0$ and $T$ only such that:
\[\forall t<T \quad \|D_xU(t,\cdot)\|_\infty\leq C_x, \|D_pU(t,\cdot)\|_\infty\leq C_p.\]
\end{prop}
\begin{proof}
Set $\xi = (\xi_1,\xi_2) \in \R^{d+m}$, $W(t,x,p,\xi) = U(t,x,p)\cdot \xi_1$ and
\begin{align*}
Z(t,x,p,\xi)&=\xi_1 \cdot \nabla_x (U\cdot \xi_1)+\xi_2\cdot\nabla_p (U\cdot \xi_1)-\beta(t)|\nabla_x U\cdot \xi_1|^2+\xi_2\cdot A\xi_2\\
&=\xi_1 \cdot \nabla_x W+ \xi_2\cdot\nabla_p W-\beta(t)|\nabla_x W|^2+\xi_2\cdot A\xi_2
\end{align*}
Remark that for $\beta_0\leq \alpha$ we have $Z|_{t = 0}\geq 0$ by assumption on $U_0$.\\ Let us now show that $Z$ stays non-negative over time for a well chosen function $\beta$. Looking at the equation satisfied by $Z$ we get
\begin{align*}
&\partial_t Z-\sigma\Delta_p Z+b\cdot \nabla_p Z+F\cdot \nabla_x Z\\
&+\left(-D_pb\xi_2+D_xb\xi_1-D_ub\nabla_x W+2\beta D_xb\nabla_xW\right)\cdot\nabla_{\xi_2} Z+\left(\nabla_xWD_uF+\nabla_p WD_ub-\xi_1D_u G\right)\cdot \nabla_{\xi_1} Z\\
&=2\beta\sigma|D^2_{xp}W|^2-\frac{d\beta}{dt}|\nabla_x W|^2\\
&+\left(\xi_1 D_xG\xi_1-\xi_1D_uG\nabla_xW-\nabla_x WD_xF\xi_1+\nabla_xWD_uF\nabla_xW\right)\\
&+\left(\nabla_xWD_pF-\xi_1D_pG\right)\cdot\xi_2+2\xi_2\left(AD_pb\xi_2+AD_xb\xi_1-AD_ub\nabla_x W\right)\\
&+2\beta\left(-2\xi_2 AD_xb\nabla_xW+\nabla_xWD_xF\nabla_xW-\xi_1D_xG\nabla_xW\right)
\end{align*}
Remark that
\begin{align*}
&\left(\xi_1 D_xG\xi_1-\xi_1D_uG\nabla_xW-\nabla_x WD_xF\xi_1+\nabla_xWD_uF\nabla_xW\right)\\
&+\left(\nabla_xWD_pF-\xi_1D_pG\right)\cdot\xi_2+2\xi_2\left(AD_pb\xi_2+AD_xb\xi_1-AD_ub\nabla_x W\right)\\
&=\left(\begin{array}{c}-\xi_1\\ \nabla_x W\\-\xi_2\end{array}\right)^T\left(\begin{array}{ccc} D_x G\quad D_u G\quad D_p G\\ D_x F \quad D_u F\quad D_pF\\2AD_x b\quad 2AD_ub\quad 2AD_p b\end{array}\right)\left(\begin{array}{c}-\xi_1\\ \nabla_x W\\-\xi_2\end{array}\right)\geq \alpha (|\xi_1|^2+|\xi_2|^2)
\end{align*}
by \eqref{(G,F,Ab)}. Hence, we end up with:
\begin{align*}
 &\partial_t Z-\sigma\Delta_p Z+b\cdot \nabla_p Z+F\cdot \nabla_x Z\\
&+\left(-D_pb\xi_2+D_xb\xi_1-D_ub\nabla_x W+2\beta D_xb\nabla_xW\right)\cdot\nabla_{\xi_2} Z+\left(\nabla_xWD_uF+\nabla_p WD_ub-\xi_1D_u G\right)\cdot \nabla_{\xi_1} Z\\
&\geq (\alpha-\beta(t)) \left(|\xi_1|^2+|\xi_2|^2)\right)+\left(4\beta\|2AD_xb\|_\infty^2+4\beta \|D_xG\|^2_\infty+2\beta\|D_xF\|_\infty-\frac{d\beta}{dt}\right)|\nabla_xW|^2.
\end{align*}
Let now $\beta_0$ be chosen so that
\[\beta_0\leq  \alpha\]
and $\beta(t)=\beta_0e^{-\lambda_\beta t}$  with
\[\lambda_\beta\geq 4\|2AD_xb\|_\infty^2+4 \|D_xG\|^2_\infty+2\|D_xF\|_\infty.\]
For such a $\beta$ we have
\begin{align*}
     &\partial_t Z-\sigma\Delta_p Z+b\cdot \nabla_p Z+F\cdot \nabla_x Z\\
&+\left(-D_pb\xi_2+D_xb\xi_1-D_ub\nabla_x W+2\beta D_xb\nabla_xW\right)\cdot\nabla_{\xi_2} Z+\left(\nabla_xWD_uF+\nabla_p WD_ub-\xi_1D_u G\right)\cdot \nabla_{\xi_1} Z\geq 0.
\end{align*}
The property $Z\geq 0$ follows from the same argument we used in Proposition \ref{prop monotonicity}. It only remains to show that this gives estimates on the gradient of the solution $U$. Remark first that by taking $\xi_2=0$, we get back
\[\forall t\in [0,T), \langle \xi_1,D_x U,\xi_1\rangle-\beta(t)\|\langle D_xU,\xi_1\rangle\|^2\geq 0.\]
As in Proposition \ref{prop monotonicity}, this implies:
\[\|D_xU(t,\cdot)\|_\infty\leq \frac{1}{\beta_t}=\frac{1}{\beta_0}e^{\lambda_\beta t}\leq \frac{1}{\beta_0}e^{\lambda_\beta T}.\]

For the gradient of $U$ in $p$, we know that \[\forall t\in [0,T),\xi_1,\xi_2,\quad \langle\xi_1,D_xU,\xi_1\rangle+\langle \xi_2,D_pU,\xi_1\rangle+\xi_2\cdot A\xi_2\geq 0.\]
Let $t< T$ be fixed, $\xi_1=xr^x_i$ and $\xi_2=yr^p_j$ where $(r_i)$ indicate basis elements.
Hence, for all $x,y \in \R$, \[ x^2\left(\frac{\partial U_i}{\partial x_i}\right)+xy\frac{\partial U_i}{\partial p_j}+A_{jj}y^2\geq 0. \]
Thus \[\left(\frac{\partial U_i}{\partial p_j}\right)^2\leq 4A_{jj}\left(\frac{\partial U_i}{\partial x_i}\right)\leq \frac{4A_{jj}}{\beta_0}e^{\lambda_\beta T}.\]
\end{proof}


\quad

\subsection{Existence of solutions}
\subsubsection{Main result}
We are now going to proceed as we did in Section 2 to justify the estimate of Proposition \ref{prop:estimate} for Lipschitz solutions of the master equation \eqref{MFG b(x,p,u)}. The heuristic presented earlier hints at the idea of studying 
\begin{equation}
\label{def Z (p,q)}
Z(t,x,y,p,q)=\langle U(t,x,p)-U(t,y,q),x-y\rangle+\langle p-q, A,p-q \rangle -\beta(t)|U(t,x,p)-U(t,y,q)|^2.
\end{equation}
Due to the doubling of variable we did in $(p,q)$, the equation satisfied for smooth $U$ (and hence smooth $Z$) is slightly different. Indeed, in this case $Z$ would be solution of:
\begin{align}
\label{eq:Tr(A)}
&\partial_t Z+F^x\cdot\nabla_x Z+F^y\cdot\nabla_y Z+b^x\cdot \nabla_p Z+b^y\cdot \nabla_q Z-\sigma \Delta_p Z-\sigma\Delta_q Z\nonumber\\
&=\langle G^x-G^y,x-y\rangle+\langle F^x-F^y,U^x-U^y\rangle+\langle 2A(b^x-b^y),p-q\rangle\\
&-2\beta \langle G^x-G^y,U^x-U^y\rangle+2\sigma\beta(|D_pU^x|^2+|D_qU^y|^2)-4\sigma TrA-\frac{d\beta}{dt}|U^x-U^y|^2\nonumber,
\end{align}
where we used the notation $U^x=U(t,x,p)$, $U^y=U(t,y,q)$, $F^x=F(x,p,U^x)$ and so on. 
However as is, equation \eqref{eq:Tr(A)} does not seem suited for the use of a comparison principle. Even under the assumption that \eqref{(G,F,Ab)} holds for some $A\in\mathcal{M}_m(\reels)$, it seems hard to get rid of the term $-4\sigma Tr(A)$ which is always negative as $A$ must be a positive semi-definite matrix. 
However, remark that
\[D_{pq} Z=-2A+2\beta D_qU^y D_pU^x.\]
Hence, for $B=\left(\begin{array}{cc} I_{m}\quad I_{m}\\ I_{m}\quad I_{m}\\\end{array}\right)$ we obtain
\[\sigma Tr(BD^2_{(p,q)} Z)=\sigma \Delta_p Z+\sigma \Delta_q Z-4\sigma TrA+4\beta \sigma Tr(D_pU^xD_qU^y)\]
where $D^2_{(p,q)}Z$ is the hessian matrix of $Z$ in $(p,q)$ only.
It follows that the equation satisfied by $Z$ might be rewritten as:
\begin{align}
\label{eq:B}
&\partial_t Z+F^x\cdot\nabla_x Z+F^y\cdot\nabla_y Z+b^x\cdot \nabla_p Z+b^y\cdot \nabla_q Z-\sigma Tr(BD^2_{(p,q)} Z)\nonumber\\
&=\langle G^x-G^y,x-y\rangle+\langle F^x-F^y,U^x-U^y\rangle+\langle 2A(b^x-b^y),p-q\rangle\\
&-2\beta \langle G^x-G^y,U^x-U^y\rangle+2\sigma\beta(|D_pU^x-D_qU^y|^2)-\frac{d\beta}{dt}|U^x-U^y|^2\nonumber.
\end{align}
As $B\geq 0$ it is now more clear that we can use a comparison principle for Equation \eqref{eq:B}. As in Section 2 we may now state the following key Lemma.

\begin{lemma}
\label{lemma: Z supersol b(x,p,U)}
Assume that $(F,G,b,U_0)$ is Lipschitz continuous. Let $U$ be a Lipschitz solution of the master equation \eqref{MFG b(x,p,u)} on $[0,T)$ for $T>0$ and $Z$ be defined as in \eqref{def Z (p,q)}. Then $Z$ is a viscosity supersolution of: 
\begin{align}
\label{Z supersol (p,q)}
&\partial_t Z+F^x\cdot\nabla_x Z+F^y\cdot\nabla_y Z+b^x\cdot \nabla_p Z+b^y\cdot \nabla_q Z-\sigma Tr(BD^2_{(p,q)} Z)\nonumber\\
&\geq \langle G^x-G^y,x-y\rangle+\langle F^x-F^y,U^x-U^y\rangle+\langle 2A(b^x-b^y),p-q\rangle\\
&-2\beta \langle G^x-G^y,U^x-U^y\rangle-\frac{d\beta}{dt}|U^x-U^y|^2,\nonumber
\end{align}
on $(0,T)\times \Omega^2\times\R^{2m}$.
\end{lemma}
\begin{proof}
We first show that $Z$ satisfies the following inequality
\begin{align}
\label{ineq Z (p,q)}
    \nonumber Z(t,x,y,p_0,q_0)&\geq \esp{Z(s,X^{x}_{t-s},Y^y_{t-s},p^{p_0}_{t-s},q^{q_0}_{t-s})-\beta(s)\left\|\int_0^{t-s}(G_u^{x,p_0}-G_u^{y,q_0})du\right\|^2}\\
    &+\esp{\int_0^{t-s}\left(\langle G_u^{x,p_0}-G_u^{y,q_0},x-y\rangle+\langle F_u^{x,p_0}-F_u^{y,q_0},U_s^{x,p_0}-U_s^{y,q_0}\rangle\right)du}\\
    \nonumber&+\esp{\int_0^{t-s}\langle p-q,2A(b_u^{x,p_0}-b_u^{y,q_0}) \rangle du}\\
    \nonumber&-(\beta(t)-\beta(s))\|U(t,x,p_0)-U(t,y,q_0)\|^2-2\beta(s)\esp{\int_0^{t-s}\langle U_s^{x,p_0}-U_s^{y,q_0},G_u^{x,p_0}-G_u^{y,q_0}\rangle},
\end{align}
with
\[
\left\{
\begin{array}{c}
     dX_s=-F(X_s,p_s,U(t-s,X_s,p_s)ds ,  \\
     dY_s=-F(Y_s,q_s,U(t-s,Y_s,q_s)ds,\\
     dp_s=-b(X_s,p_s,U(t-s,X_s,p_s)ds+\sqrt{2\sigma}dW_s,\\
     dq_s=-b(Y_s,q_s,U(t-s,Y_s,q_s)ds+\sqrt{2\sigma}dW_s,\\
     U^{x,p_0}_s=U(t-s,X_s^x,p_s^{p_0}),\\
     G_s^{x,p_0}=G(X_s^x,p_s^{p_0},U_s^{x,p_0}),
\end{array}
\right.
\]
and the notation $F_s^{x,p_0}$ and $b_s^{x,p_0}$ being analogous to the ones for G. As usual in this kind of coupling techniques, or in comparison questions in the theory of viscosity solutions, it is fundamental for the two processes $(p_s)_{s\geq 0},(q_s)_{s\geq 0}$ to be generated with the same Brownian motion. Indeed if the computations were done for two independent Brownian motions, we would end up with $Z$ being a supersolution of \eqref{eq:Tr(A)} (i.e there will be a term in $Tr(A)$). The link between \eqref{eq:Tr(A)} and \eqref{ineq Z (p,q)} is obvious for smooth $Z$ but it does not seems so easy to switch from one to the other in the viscosity sense. 

\quad

Let 
\[W(t,x,y,p,q)=\langle U(t,x,p)-U(t,y,q),x-y\rangle.\]
It follows that 
\begin{align*}
    Z(t,x,y,p,q)=\underbrace{W(t,x,y,p,q)}_{I}+&\underbrace{\langle p-q, A(p-q)\rangle}_{II}-\beta(s)\underbrace{|U(t,x,p)-U(t,y,p)|^2}_{III}\\
    &-(\beta(t)-\beta(s))|U(t,x,p)-U(t,y,p)|^2 .
\end{align*} 

For $II$, we can write
\begin{align*}
    \langle p_0-q_0, A,p_0-q_0\rangle&=\esp{\langle p^{p_0}_{t-s}-q^{q_0}_{t-s}, A,p^{p_0}_{t-s}-q^{q_0}_{t-s}\rangle+\int_0^{t-s}\langle p-q,A+A^T,b_u^{x,p_0}-b_u^{y,q_0} \rangle du}\\
    &+\esp{\langle \int_0^{t-s}(p_u^{p_0}-q_u^{q_0})du,A,\int_0^{t-s}(p_u^{p_0}-q_u^{q_0})du\rangle}\\
    &\geq \esp{\langle p^{p_0}_{t-s}-q^{q_0}_{t-s}, A,p^{p_0}_{t-s}-q^{q_0}_{t-s}\rangle+\int_0^{t-s}\langle p-q,A+A^T,b_u^{x,p_0}-b_u^{y,q_0} \rangle du}.
\end{align*}
The terms $I$ and $III$ are treated as in Lemma \ref{lemma Z supersol b(p)}. Hence, because we assumed that $A$ is symmetric, using the inequalities satisfied by those 3 terms, $Z$ satisfies \eqref{ineq Z (p,q)}.

As in the autonomous case, considering a test function $\varphi$ such that
\[(Z-\varphi)(t,x,y,p,q)=\min(Z-\varphi)=0,\] we get that at $(t,x,y,p,q)$, $\varphi$ satisfies
\begin{align*}
&\partial_t \varphi+F^x\cdot\nabla_x \varphi+F^y\cdot\nabla_y \varphi+b^x\cdot \nabla_p \varphi+b^y\cdot \nabla_q \varphi-\sigma Tr(BD^2_{(p,q)} \varphi)\\
&\geq \langle G^x-G^y,x-y\rangle+\langle F^x-F^y,U^x-U^y\rangle+\langle 2A(b^x-b^y),p-q\rangle\\
&-2\beta \langle G^x-G^y,U^x-U^y\rangle-\frac{d\beta}{dt}|U^x-U^y|^2.
\end{align*}
Observe that the term $\sigma Tr(BD^2_{(p,q)} \varphi)$ appears as a direct consequence of Ito's lemma because we used the same Brownian motion for $p$ and $q$. 
\end{proof}

Armed with this Lemma, we may now state a global existence result for Lipschitz solutions of the master equation \eqref{MFG b(x,p,u)}. We recall that there is always uniqueness of a maximal Lipschitz solution so it will imply a well-posedness result.
 \begin{thm}
 \label{theorem:existence b(x,p,u)}
     Assume that $(U_0,G,F,b)$ is Lipschitz continuous in all the variables and further assume that Hypothesis \ref{hyp:b(x,p,U)} holds for some $\alpha>0$ and $A\in \mathcal{M}_m(\reels)$.
    Then there exists a unique global Lipschitz solution of the master equation \eqref{MFG b(x,p,u)}.
 \end{thm}
\begin{proof}
Since $(U_0,G,F,b)$ is Lipschitz continuous, we have existence of a Lipschitz solution at least on $[0,T_c)$ for some $T_c>0$. Let $T<T_c$, thanks to Lemma \ref{lemma: Z supersol b(x,p,U)}, we already know that
\[
Z(t,x,y,p,q)=\langle U(t,x,p)-U(t,y,q),x-y\rangle+\langle p-q, A,p-q \rangle -\beta(t)|U(t,x,p)-U(t,y,q)|^2
\]
is a viscosity supersolution of Equation \eqref{Z supersol (p,q)} on $(0,T)$. It only remains to choose a smooth function $\beta$ such that both $Z_{t=0}\geq 0$ and 
\begin{align*}
\langle G^x-G^y,x-y\rangle+\langle F^x-F^y,U^x-U^y\rangle+\langle &2A(b^x-b^y),p-q\rangle\\
-2\beta \langle G^x-G^y,U^x-U^y\rangle-\frac{d\beta}{dt}|U^x-U^y|^2\geq 0
\end{align*}
holds. The fact that $(F,G,b)$ satisfies \eqref{(G,F,Ab)} yields
\begin{align*}
&\langle G^x-G^y,x-y\rangle+\langle F^x-F^y,U^x-U^y\rangle+\langle 2A(b^x-b^y),p-q\rangle-2\beta \langle G^x-G^y,U^x-U^y\rangle-\frac{d\beta}{dt}|U^x-U^y|^2\\
&\geq \alpha \left(|x-y|^2+|p-q|^2\right)-\frac{d\beta}{dt}|U^x-U^y|^2-2\beta \langle G^x-G^y,U^x-U^y\rangle\\
&\geq (\alpha-\beta )\left(|x-y|^2+|p-q|^2\right)-\left(\frac{d\beta}{dt}+\beta (2\|D_uG\|_\infty+\|D_xG\|^2_\infty+\|D_pG\|^2_\infty)\right)|U^x-U^y|^2.
\end{align*}
\color{black}
We also know that taking $\beta_0\leq \alpha$ implies that $Z_{t=0}\geq 0$ as $U_0$ satisfies \eqref{condition U_0 V lineaire }. For $\beta(t)=\beta_0e^{-\lambda_\beta t}$ with $\lambda_\beta\geq  2\|D_uG\|_\infty+\|D_xG\|^2_\infty+\|D_pG\|^2_\infty$ and $\beta_0\leq \alpha$ both inequalities are satisfied. 

Finally we get that $Z$ is a viscosity supersolution of 
\[\partial_t Z+F^x\cdot\nabla_x Z+F^y\cdot\nabla_y Z+b^x\cdot \nabla_p Z+b^y\cdot \nabla_q Z-\sigma Tr(BD^2_{(p,q)} Z)\geq 0,\]
on $(0,T)$ with $Z_{t=0}\geq 0$. 
We may then conclude that $Z$ stays non-negative on $[0,T)$ with Theorem 2.5 of \cite{34bed8a7-6223-3f55-b20d-c7daa21c8c15}.

Evaluating $Z$ for $p=q$ we recover an estimate on the Lipschitz norm of $U$ in $x$ on $[0,T)$ as in the autonomous case. For $\|D_pU\|_\infty$ we know that
\[\langle U(t,x,p)-U(t,y,q),x-y\rangle+\langle p-q,A,p-q\rangle\geq 0.\]
Taking $x=y+\lambda e_i$ and using the fact that $U$ is Lipschitz in $x$ we get 
\[\lambda^2\|D_xU\|_\infty+\lambda (U^i(t,x,p)-U^i(t,x,q))+\langle p-q,A,p-q\rangle\geq 0,\]
which implies that
\[\forall (t,x,p,q)\in [0,T)\times \Omega\times \reels^{2m}, \quad (U^i(t,x,p)-U^i(t,x,q))^2-4\|D_xU\|_\infty\langle p-q,A,p-q\rangle\leq 0.\]
Denoting by $\lambda_a>0$ be the biggest eigenvalue of $A$, we obtain
\[\forall t<T, \quad \|D_pU(t,\cdot)\|_\infty\leq 2\sqrt{\|D_xU(t,\cdot)\|_\infty \lambda_a}.\]

Suppose now that the maximum time of existence satisfies $T_c<+\infty$. By the local existence Theorem 2.4 of \cite{bertucci2023lipschitz} we must have
\[\underset{t\to T_c}{\lim}\|D_{(x,p)}U(t,\cdot)\|_\infty=+\infty.\]
However we have an uniform bound in time on the Lipschitz norm of $U$ in $(x,p)$ for $t<T_c$. As a consequence this is absurd, hence $T_c=+\infty$ and the result is proved.
\end{proof}
\begin{remarque}
Take note that when $D_xb,D_ub=0$ we recover the hypothesis for the autonomous case. Let $Z$ be defined as in Section 2 by
\[Z(t,x,y,p)=\langle U(t,x,p)-U(t,y,p),x-y\rangle -\beta(t)|U(t,x,p)-U(t,y,p)|^2.\]
Define
\[M=\underset{[0,T]\times \Omega^2\times \reels^{2m}}{\min} Z,\]
and consider
\[Z^\varepsilon(t,x,y,p,q)=\langle U(t,x,p)-U(t,y,q),x-y\rangle -\beta(t)|U(t,x,p)-U(t,y,q)|^2+\frac{1}{2\varepsilon}|p-q|^2+\alpha|p|^2+\alpha|q|^2.\]
Looking at point of minimum of $Z^\varepsilon$ we can show that as $\varepsilon$ tends to 0 there exists a subsequence $(t_{\varepsilon_n},x_{\varepsilon_n},y_{\varepsilon_n},p_{\varepsilon_n},q_{\varepsilon_n})_{n\geq 0}$ such that 
\[
\left\{
\begin{array}{c}
    (t_{\varepsilon_n},x_{\varepsilon_n},y_{\varepsilon_n},p_{\varepsilon_n},q_{\varepsilon_n})\to (t^*,x^*,y^*,p^*,q^*)\in [0,T]\times \Omega^2\times \reels^{2m} \\
    \underset{n\to+\infty }{\lim}Z^{\varepsilon_n}(t_{\varepsilon_n},x_{\varepsilon_n},y_{\varepsilon_n},p_{\varepsilon_n},q_{\varepsilon_n})=M\\
    \frac{1}{\varepsilon_n}|p_{\varepsilon_n}-q_{\varepsilon_n}|^2\to 0\\
    |p_{\varepsilon_n}-q_{\varepsilon_n}|\to 0\\
\end{array}
\right.
\]
Because $\frac{1}{\varepsilon}\langle b(p)-b(q),p-q\rangle \leq \frac{C}{\varepsilon}|p-q|^2$, we may conclude by taking first $\varepsilon_n\to 0$ and then $\alpha\to 0$ that we only need conditions \eqref{condition U_0 V lineaire } and \eqref{(G,F,Ab)} to hold for $p=q$ to deduce that $M\geq 0$. This way, we indeed recover the hypothesis we had in the autonomous case. 
\end{remarque}

\subsection{Extensions}
Now that we have established the main setting, let us discuss on some extensions
\subsubsection{On the strong monotonicity of $(G,F,b)$}
As we stated it at the beginning the strong monotonicity assumptions on $(G,F,b)$, condition \eqref{(G,F,Ab)} can be weakened quite a bit. Following the proof of Theorem \ref{theorem:existence b(x,p,u)}, the term we control with strong monotonicity depends on $G$ only. This means that strong monotonicity is only needed whenever $G$ is not degenerate. In particular if $G$ does not depend on $p$ then strong monotonicity in this variable is not needed. We may summarize this observation in the following corollary:
\begin{corol}
\label{corol: weaker monotonicity}
Suppose there exist a matrix $A\in\mathcal{M}_m(\reels)$ and $\alpha>0$ such that for all $x,y$ in $\Omega$, $p,q$ in $\reels^m$ and $u,v$ in $\reels^d$ there is $t=t(x,y,p,q,u,v)\in[0,1]$ such that
\begin{align*}
\langle U_0(x,p)-U_0(y,q),x-y\rangle+&\langle p-q,A(p-q)\rangle\geq \alpha|U_0(x,p)-U_0(y,q)|^2,\\
\nonumber\\
\nonumber\langle G(x,p,u)-G(y,q,v),x-y&\rangle+\langle F(x,p,u)-F(x,q,v),u-v\rangle+\langle b(x,p,u)-b(y,q,v),(A+A^T)(p-q)\rangle \\
&  \geq \alpha |G(x,p,tu+(1-t)v)-G(y,q,tu+(1-t)v)|^2.
\end{align*}
Then there exist a Lipschitz solution to \eqref{MFG b(x,p,u)} on any time interval. 
\end{corol}
\begin{proof}
    The only difference lies in how we deal with the term 
    \[\langle G(x,p,u)-G(y,q,v),u-v\rangle.\]
    Observe that there is no difficulty in adapting the proof of Theorem \ref{theorem:existence b(x,p,u)} as for any $t\in[0,1]$, 
    \[\langle G(x,p,u)-G(y,q,v),u-v\rangle\leq \frac{1}{2}|G(x,p,tu+(1-t)v)-G(y,q,tu+(1-t)v)|^2+\left(\frac{1}{2}+\|D_uG\|_\infty\right)|u-v|^2. \]
\end{proof}
\begin{remarque}
The addition of this $t\in[0,1]$ in the monotonicity condition might seems cryptic at first. Essentially it is a way of translating the fact that we do not need strong monotonicity along the variable associated to $U$.
\end{remarque}

Similarly, the gradient estimate of Proposition \ref{prop:estimate} may be stated under weaker monotonicity assumptions. Following the proof it is evident that we only really need $(G,F,b)$ to satisfy
\[\forall \nu\in\reels^{2d+m}\quad \nu^T\left(\begin{array}{ccc} D_x G\quad D_u G\quad D_p G\\ D_x F \quad D_u F\quad D_pF\\2AD_x b\quad 2AD_ub\quad 2AD_p b\end{array}\right)\nu\geq \alpha \left|\left(\begin{array}{ccc} D_x G& 0 &0\\ 0 & 0 &0\\0 &0 &D_xb\end{array}\right)\nu\right|^2,\]
for some $\alpha>0$ and a symmetric matrix $A$. Obviously, this condition is not equivalent to the one we stated in Corollary \ref{corol: weaker monotonicity}. We believe the reason for that to be sligtly technical, indeed in general condition \eqref{condition U_0 V lineaire } 
\[\langle U_0(t,x,p)-U_0(t,y,q),x-y\rangle+\langle p-q,A,p-q\rangle \geq \alpha |U_0(t,x,p)-U_0(t,y,q)|^2\]
and \eqref{U0 V lineaire diff} 
\[\forall (\xi_1,\xi_2)\quad \left(\begin{array}{c}\xi_1\\ \xi_2\end{array}\right)^T\left(\begin{array}{cc} D_xU_0\quad D_p U_0\\ 0\quad A\\\end{array}\right)\left(\begin{array}{c}\xi_1\\ \xi_2\end{array}\right)\geq \alpha|D_x U_0\cdot \xi_1|^2,\]
are not equivalent even though we believe them both to be natural assumptions. They are however both sufficient condition  to recover existence and uniqueness of solution to the master equation under the strong monotonicity condition \eqref{(G,F,Ab)}. Indeed notice that that whenever \eqref{U0 V lineaire diff} is fulfilled for some $(\alpha,A)$, for any $\varepsilon>0$ there exist an $\alpha_\varepsilon>0$ such that \eqref{condition U_0 V lineaire } holds for the couple $(\alpha_\varepsilon,\varepsilon I_d+A)$. The idea being that whenever strong monotonicity of $(G,F,Ab)$ holds, the addition of this small perturbation $\varepsilon$ is non consequential and we may recover bounds on the gradient as in the proof of Theorem \ref{theorem:existence b(x,p,u)} without much additional difficulty.

\color{black}
\subsubsection{A uniqueness result}
Let us first present a uniqueness property. It may seem surprising that we are now concerned with uniqueness while we insisted that it always holds in the class of Lipschitz solutions. Even though in this context uniqueness follows from the Lipschitz regularity of solutions we provide a proof of uniqueness to indicate why developments similar to monotone solutions \cite{JEP_2021__8__1099_0} are also immediate for equation \eqref{MFG b(x,p,u)}. Monotone solutions only require the continuity of the value function U and yields uniqueness and stability properties under monotonicity assumptions on the coefficients. 

\begin{lemma}
\label{lemma: uniqueness b(x,p,U)}
Assume that $U_0$ and $(G,F,b)$ satisfy:
\[
\begin{array}{c}
\exists A\in\mathcal{M}(\reels^m),\quad \forall (x,y,u,v,p,q)\in \Omega^2\times \reels^{2d}\times \reels^{2m},\\
\\
\langle U_0(x,p)-U_0(y,q),x-y\rangle+\langle p-q,A(p-q)\rangle\geq 0,\\
\\
\langle G(x,p,u)-G(y,q,v),x-y\rangle+\langle F(x,p,u)-F(x,q,v),u-v\rangle+\langle b(x,p,u)-b(y,q,v),A(p-q)\rangle \\
   \geq 0.
\end{array}
\]
Then there exists at most one smooth solution of the master equation \eqref{MFG b(x,p,u)}.
\end{lemma}

\begin{proof}
Suppose there exists two smooth solutions $U$ and $V$ on $[0,T]$ for some $T>0$ and take
\[Z(t,x,y,p,q)=\langle U(t,x,p)-V(t,y,q),x-y\rangle +\langle p-q,A,p-q\rangle.\]
First, the assumption formulated on $U_0$ ensure that $Z_0\geq 0$. Since U and V are smooth, so is $Z$. Using the equations satisfied by $U$ and $V$, we get that $Z$ is solution of 
\begin{align*}
&\partial_t Z+F^x\cdot\nabla_x Z+F^y\cdot\nabla_y Z+b^x\cdot \nabla_p Z+b^y\cdot \nabla_q Z-\sigma Tr(BD^2_{(p,q)} Z)\nonumber\\
&=\langle G^x-G^y,x-y\rangle+\langle F^x-F^y,U^x-V^y\rangle+\langle 2A(b^x-b^y),p-q\rangle,
\end{align*}
where we used the notation $G^x=G(x,p,U(t,x,p))$, $G^y=G(y,q,V(t,y,q))$, similarly for $F^x,F^y,b^x,b^y$ and $B=\left(\begin{array}{cc} I_{m}\quad I_{m}\\ I_{m}\quad I_{m}\\\end{array}\right)$.\\
Finally, the monotonicity condition on $(G,F,Ab)$ implies (recall that $Z$ is smooth) that $Z$ satisfies
\[
\left\{
\begin{array}{c}
\partial_tZ+F^x\cdot \nabla_x Z+F^y\nabla_y Z+b^x\nabla_p Z+b^y\nabla_q Z-\sigma Tr(BD^2_{(p,q)} Z)\geq 0\quad  \forall t\in(0,T),\\
Z_{t=0}\geq 0.
\end{array}
\right. 
\]
By a comparison principle with 0 we deduce that $Z\geq 0$. Fixing $p=q\in\reels^m$ yields that for all $t \geq 0, x,y \in \Omega$
\[\langle U(t,x,p)-V(t,y,p),x-y\rangle \geq 0.\]
Take $x\in \overset{\circ}{\Omega}$ and set $y=x+\varepsilon\xi$ with $\|\xi\|=1$ such that for $\varepsilon$ sufficiently small $y\in \Omega$. Dividing by $\varepsilon$ and taking the limit as $\varepsilon$ tends to $0$ in the previous inequality gives
\[ \forall \|\xi\|\leq 1 \quad \langle U(t,x,p)-V(t,x,p),\xi \rangle \geq 0,\]
which implies $U(t,x,p) = V(t,x,p)$. By continuity the equality holds for any $x \in \Omega$.
\end{proof}

\subsubsection{Recovering monotonicity estimates under different assumptions}

It was first observed in \cite{P_L_Lions} that we may trade the $\alpha-$monotonicity of $U_0$ against stronger assumptions on the coefficients of the equation. Following this idea we state the following Theorem.
\begin{thm}
\label{estimate alternatif}
    Suppose $(U_0,G,F,b)$ is Lipschitz continuous. Further assume that
    \begin{align}
\nonumber \forall (x,y,p,q,&u,v)\in \Omega^2\times \reels^{2m}\times \reels^{2d},\\
\label{U0 A monotone }\langle U_0(x,p)-U_0(y,q),x&-y\rangle+\langle p-q,A(p-q)\rangle\geq 0,\\
\nonumber\\
\nonumber\langle G(x,p,u)-G(y,q,v),x-y\rangle+\langle F(x,&p,u)-F(x,q,v),u-v\rangle+\langle b(x,p,u)-b(y,q,v),2A(p-q)\rangle \\
\label{(G,F,Ab) en u}    &  \geq \alpha |u-v|^2.
\end{align}
Then there exists a global Lipschitz solution.
\end{thm}
\begin{proof}
We only provide a sketch of the proof as we believe the extension to be natural. Define 
\begin{align*}
Z(t,x,y,p,q)=\langle U(t,x,p)-U(t,y,q),&x-y\rangle+\langle p-q, A,p-q \rangle -\beta(t)|U(t,x,p)-U(t,y,q)|^2\\
&+\gamma(t)(|x-y|^2+|p-q|^2).
\end{align*}
By a proof similar to the one of Lemma \ref{lemma: Z supersol b(x,p,U)}, $Z$ is a viscosity supersolution of 
\begin{align*}
&\partial_t Z+F^x\cdot\nabla_x Z+F^y\cdot\nabla_y Z+b^x\cdot \nabla_p Z+b^y\cdot \nabla_q Z-\sigma Tr(BD^2_{(p,q)} Z)\nonumber\\
&\geq \langle G^x-G^y,x-y\rangle+\langle F^x-F^y,U^x-U^y\rangle+\langle 2A(b^x-b^y),p-q\rangle\\
&-2\beta \langle G^x-G^y,U^x-U^y\rangle-\frac{d\beta}{dt}|U^x-U^y|^2\nonumber\\
&+2\gamma(t)( \langle F^x-F^y,x-y\rangle+\langle (b^x-b^y),p-q\rangle)+\frac{d\gamma}{dt}(|x-y|^2+|p-q|^2).
\end{align*}
For $ \max(\|D_xU_0\|^2_\infty,\|D_pU_0\|^2_\infty)\beta_0\leq  \gamma_0$, \eqref{U0 A monotone } implies that $Z_{t=0}\geq 0$. Using \eqref{(G,F,Ab) en u} we know that if we can find a couple $(\beta,\gamma)$ such that
\[
\left\{
\begin{array}{c}
\frac{d\gamma}{dt}\geq 2\gamma (\|F\|_{Lip}+\|b\|_{Lip})+\beta \|G\|_{Lip},\\
\alpha \geq \frac{d\beta}{dt}+2\beta \|G\|_{Lip}+\gamma (\|F\|_{Lip}+\|b\|_{Lip}),
\end{array}
\right.
\]
on $(0,T)$ for the Lipschitz semi-norm $\|F\|_{Lip}=\|D_xF\|_\infty+\|D_pF\|_\infty+\|D_uF\|_\infty$, then $Z$ will be a viscosity supersolution of 
\[\partial_t Z+F^x\cdot\nabla_x Z+F^y\cdot\nabla_y Z+b^x\cdot \nabla_p Z+b^y\cdot \nabla_q Z-\sigma Tr(BD^2_{(p,q)} Z)\geq 0\]
on $(0,T)\times\Omega^2\times \R^{2m}$. To find such a couple, it suffices to take $\beta(t)=\frac{1}{\max(\|D_xU_0\|^2_\infty,\|D_pU_0\|^2_\infty)}\gamma(t)$ with 
\[\gamma(t)=\gamma_0 e^{\lambda(t-T)},\]
for $\lambda$ sufficiently big and $\gamma_0$ sufficiently small chosen independently of $T$. We may then recover $Z\geq 0$ on $[0,T)$ by a comparison principle, which of course implies an estimate on the Lipschitz norm of $U$ in $(x,p)$. 
\end{proof}

\subsubsection{The case of a non-constant volatility} \label{section: vol}
Let us also comment on monotonocity estimates for the more general equation
\begin{align}
\label{MFG b(x,p,u) vol}
    \begin{array}{c}
      \partial_t U+F(x,p,U)\cdot \nabla_x U+b(x,p,U)\cdot \nabla_p U-Tr(\Gamma(x,p,U)D^2_pU) = G(x,p,U) \text{ in }(0,T)\times\Omega\times\reels^{m}, \\
      U(0,x,p)=U_0(x,p) \text{ in }\Omega\times\reels^{m},
    \end{array}
\end{align}
for $\Gamma(x,p,U)=\Sigma(x,p,U)\Sigma^T(x,p,U)$ with $\Sigma: \Omega\times \reels^m\times \reels^d\to \mathcal{M}_m(\reels)$. Let us first remark that when $\Sigma$ is Lipschitz in all variables, we may still define a Lipschitz solution of equation \eqref{MFG b(x,p,u) vol}. Let us now state a global existence result for Lipschitz solution of this equation

\begin{prop}
Suppose $U_0,F,G,b,\Sigma$ are Lipschitz continuous in all variables. If they furthermore satisfy the following
\begin{align}
\nonumber &\forall (x,y,p,q,u,v)\in \Omega^2\times \reels^{2m}\times \reels^{2d}\\
\nonumber &\langle U_0(x,p)-U_0(y,q),x-y\rangle+\langle p-q,A(p-q)\rangle\geq \alpha|U_0(x,p)-U_0(y,q)|^2,\\
\nonumber&\langle G(x,p,u)-G(y,q,v),x-y\rangle+\langle F(x,p,u)-F(x,q,v),u-v\rangle+\langle b(x,p,u)-b(y,q,v),2A(p-q)\rangle, \\
\label{(G,F,Ab) vol}    &  \geq \alpha (|x-y|^2+|p-q|^2)+Tr\left( \left(\Sigma(x,p,u)-\Sigma(y,q,v)\right)^T 2A \left(\Sigma(x,p,u)-\Sigma(y,q,v)\right)\right),
\end{align}
for a matrix $A\in\mathcal{M}_m(\reels)$, then there exists a global Lipschitz solution of equation \eqref{MFG b(x,p,u) vol} on any time interval. 
\end{prop}
\begin{proof}
We only sketch the proof as it follows from the one of Theorem \ref{theorem:existence b(x,p,u)}. Defining $Z$ as in this Theorem, for a smooth $U$, it is a solution of
\begin{align*}
&\partial_t Z+F^x\cdot\nabla_x Z+F^y\cdot\nabla_y Z+b^x\cdot \nabla_p Z+b^y\cdot \nabla_q Z-Tr(\Gamma^xD^2_p Z)-Tr(\Gamma^y D^2_qZ)\nonumber\\
&=\langle G^x-G^y,x-y\rangle+\langle F^x-F^y,U^x-U^y\rangle+\langle 2A(b^x-b^y),p-q\rangle\\
&-2\beta \langle G^x-G^y,U^x-U^y\rangle-\frac{d\beta}{dt}|U^x-U^y|^2\nonumber\\
&+2\beta\left(|\Sigma^x D_p U^x|^2+|\Sigma^y D_q U^y|^2\right)-2Tr((\Gamma^x+\Gamma^y)A),
\end{align*}
with the notation $\displaystyle |M|^2=Tr(M^TM)=\sum_{i,j} M_{ij}^2$ for a square matrix and the same notation $U^x,F^x$ and so on we used in earlier proofs. Remark that
\[D_p (D_qZ)=-2A+2\beta (D_pU^x)^TD_qU^y.\]
We are also going to need the formula
\[(\Gamma^x+\Gamma^y)A=(\Sigma^x-\Sigma^y)(\Sigma^x-\Sigma^y)^T A+\left(\Sigma^x(\Sigma^y)^T+\Sigma^y(\Sigma^x)^T\right)A.\]
Thanks to those two equations, we may rewrite the PDE satisfied by $Z$ as 
\begin{align*}
&\partial_t Z+F^x\cdot\nabla_x Z+F^y\cdot\nabla_y Z+b^x\cdot \nabla_p Z+b^y\cdot \nabla_q Z-Tr(\Lambda D^2_{(p,q)}Z)\nonumber\\
&=\langle G^x-G^y,x-y\rangle+\langle F^x-F^y,U^x-U^y\rangle+\langle 2A(b^x-b^y),p-q\rangle\\
&-2\beta \langle G^x-G^y,U^x-U^y\rangle-\frac{d\beta}{dt}|U^x-U^y|^2\nonumber\\
&+2\beta |\Sigma^x D_pU^x-\Sigma^yD_qU^y|^2-2Tr\left((\Sigma^x-\Sigma^y)^T A(\Sigma^x-\Sigma^y))\right),
\end{align*}
for $\Lambda=\left(\begin{array}{cc} \Gamma^x\quad \Sigma^x(\Sigma^y)^T\\ \Sigma^y(\Sigma^x)^T\quad \Gamma^y\\\end{array}\right)\geq 0.$

Adapting Lemma \ref{Z supersol (p,q)}, we may prove that when $U$ is not smooth, $Z$ is still a viscosity supersolution of 
\begin{align*}
&\partial_t Z+F^x\cdot\nabla_x Z+F^y\cdot\nabla_y Z+b^x\cdot \nabla_p Z+b^y\cdot \nabla_q Z-Tr(\Lambda D^2_qZ)\nonumber\\
&\geq \langle G^x-G^y,x-y\rangle+\langle F^x-F^y,U^x-U^y\rangle+\langle 2A(b^x-b^y),p-q\rangle\\
&-2\beta \langle G^x-G^y,U^x-U^y\rangle-\frac{d\beta}{dt}|U^x-U^y|^2\nonumber-2Tr\left((\Sigma^x-\Sigma^y)^T A(\Sigma^x-\Sigma^y))\right).
\end{align*}
It then only remains to use the assumption \eqref{(G,F,Ab) vol} and to conclude for a well chosen $\beta$. 
\end{proof}

\begin{remarque}
    This shows that when $\Sigma$ depends on $(x,p)$ only, there always exists a global Lipschitz solution under Hypothesis \ref{hyp:b(x,p,U)} for $\|\Sigma\|_{Lip}$ sufficiently small. 
\end{remarque}
\subsubsection{Stationary Master equation}
We could add a discount term $\lambda U$ in equation \eqref{MFG b(x,p,u)} and the analysis would stay unchanged. In fact such term is regularizing for the equation as soon as $\lambda>0$. However the presence of such factor seems to be a necessary condition to be able to talk about steady states of equation \eqref{MFG b(x,p,u)} and the associated hypoelliptic PDE
\begin{equation}
\label{stationary MFG b(x,p,u)}
\begin{array}{c}
\lambda U+F(x,p,U)\cdot \nabla_x U+b(x,p,U)\cdot \nabla_p U-\sigma \Delta_p U=G(x,p,U) \text{ for }   x \in \Omega,p \in\R^m.
\end{array}
\end{equation}
Without the additional variable $p$, Equation \eqref{stationary MFG b(x,p,u)} was studied in \cite{JEP_2021__8__1099_0}. Let us also briefly mention that this equation has intrinsic significance independently of its interpretation as the large time limit of the solution of its parabolic counterpart. Indeed, it has been observed that the mean field limit of general economic equilibrium may fall into this category of equation \cite{bertucci2023singular,bertucci2020mean,achdou2022class}.
Just as we did in the proof of Lemma \ref{lemma: uniqueness b(x,p,U)}, it is possible to obtain uniqueness of solutions to this equation by using the auxiliary function
\[Z(x,p)=\langle U(x,p)-U(y,q),x-y\rangle+\langle A(p-q),p-q\rangle.\]
We do not details the proof all over again as it follows naturally from previously introduced arguments. However, contrary to the parabolic case, there seems to be no other way than looking at the equation satisfied by the gradient of U to get Lipschitz estimates. We believe this may be done without much difficulties by taking inspiration from \cite{JEP_2021__8__1099_0} and Proposition \ref{prop:estimate} so we do not prove it here. Let us however state this result for the associated parabolic PDE and explain how it relates to long time convergence. 
\begin{lemma}
\label{lemma: boundedness lambda}
    Under hypothesis \ref{hyp:b(x,p,U)} there exist a Lipschitz solution U of the following equation
    \begin{equation}
\label{MFG b(x,p,u) lambda}
\left\{
\begin{array}{c}
\partial_t U+\lambda U+F(x,p,U)\cdot \nabla_x U+b(x,p,U)\cdot \nabla_p U-\sigma \Delta_p U=G(x,p,U) \text{ for }  t\in(0,T), x \in \Omega,p \in\R^m, \\
U(0,x,p)=U_0(x,p) \text{ for } (x,p)\in \Omega\times \reels^m,
\end{array}
\right. 
\end{equation}
on any time interval. Moreover, if we assume that $\lambda\geq 2\|D_uG\|_\infty$, that $U_0$ is bounded and that $G$ has linear growth in $U$ then there exist a constant $C$ which depends on the data of the problem such that
\[\forall t<+\infty \quad  \|U(t,\cdot)\|_\infty+\|(D_xU,D_pU)(t,\cdot)\|_\infty\leq C.\]
\end{lemma}
\begin{proof}
The Lipschitz estimate follows quite easily from the proof of Theorem \ref{theorem:existence b(x,p,u)}. It suffices to notice that because $\lambda\geq 2\|D_uG\|_\infty$ we may take a constant $\beta$ which gives exactly this uniform in time Lipschitz estimate. 
The bound on $\|U\|_\infty$ is obtained by mean of Gronwall Lemma whenever $U_0$ is bounded and $G$ has linear growth in $U$. 
\end{proof}
\begin{remarque}
Strong monotonicity conditions on $(F,G,b)$ in $(x,p)$ only as in Hypothesis \ref{hyp:b(x,p,U)} do not in general help for long time convergence: the condition $\lambda\geq2\|D_uG\|_\infty$ appears to be necessary to control the Lipschitz norm of $U$ uniformly in time if no strong monotonicity of the system is imposed in $U$. Consider the following example in dimension 1
    \begin{equation*}
    \left\{
        \begin{array}{c}
        \lambda>0,\\
             F(x,U)=-\gamma x,  \\
             G(x,U)=\alpha x+\gamma U.
        \end{array}
        \right.
    \end{equation*}
    Obviously for this choice
    \[\left(\begin{array}{cc}
         D_x G &D_u G  \\
         D_xF & D_uF
    \end{array}\right)\geq \alpha \left(\begin{array}{cc}
         1 &0  \\
         0 & 0
    \end{array}\right), \]
    in the order of positive semi-definite matrix. However, the solution of the associated PDE on $\reels$
    \[\partial_t U+\lambda U-\gamma x\cdot \nabla_x U=\alpha x+\gamma U,\]
    is given by 
    \[U(t,x)=e^{(\gamma-\lambda)t}U_0(xe^{\gamma t})+\frac{\alpha}{\lambda}\left(e^{(2\gamma-\lambda)t}-e^{(\gamma-\lambda)t}\right)x.\]
    It is evident in this particular example that there is no hope to recover uniform in time Lispchitz estimate unless the condition
    \[\lambda\geq 2\gamma,\]
    is fulfilled. Under stronger assumptions on the monotonicity of $(G,F,b)$ in $U$ the conclusion of Lemma \ref{lemma: boundedness lambda} may hold even for $\lambda=0$ (though we expect that $\lambda>0$ might be necessary to recover uniqueness of the limit). We do not mean to be exhaustive here and leave such developments to the interested reader. 
\end{remarque}
We believe this Lemma to be almost sufficient just by itself to conclude to the converge of a solution $U$ of \eqref{MFG b(x,p,u) lambda} to a solution of \eqref{stationary MFG b(x,p,u)} as time tends to infinity. Indeed it gives compactness of the sequence $(V_n)_{n\in\entiers}$ defined by
\[ V_n(t,x,p)=U(t+n,x,p) \text{ for } (t,x,p)\in [0,1]\times\Omega\times\reels^m.\]
Along a subsequence, $(V_n)$ converges locally uniformly to a Lipschitz bounded function $V$ which does not depend on time. Now, Lipschitz solution are not adapted, at least in the form we introduced them in this paper, to equation \eqref{stationary MFG b(x,p,u)}. However we believe viscosity like information to be sufficient to characterise the limit function $V$ as a weak solution of equation \eqref{stationary MFG b(x,p,u)} by adapting monotone solutions \cite{JEP_2021__8__1099_0} to equation \eqref{MFG b(x,p,u) lambda} with the help of Lemma \ref{lemma: uniqueness b(x,p,U)}. The property of stability enjoyed by those solution would guarantee that the limiting function $V$ is a monotone solution of \eqref{stationary MFG b(x,p,u)} while uniqueness would ensure the sequence $(V_n)$, and hence $U$, converges indeed to $V$.
\color{black}

\subsubsection{Other notions of monotonicity}
It is an observation of P.-L. Lions presented in  \cite{JEP_2021__8__1099_0} that instead of monotonicity of $U_0$ in $x$ we may ask for monotonicity along a function $\phi:\reels^d\to \reels^d$:
\[\langle U_0(x,p)-U_0(y,p),\phi(x)-\phi(y)\rangle \geq 0.\]
In which case, conditions \eqref{condition U_0 V lineaire } and \eqref{(G,F,Ab)} may become for instance
\[
\begin{array}{c}
\forall (x,y,p,q,u,v)\in \Omega^2\times\reels^{2m}\times \reels^{2d},\\
\\
      \langle U_0(x,p)-U_0(y,q),\phi(x)-\phi(y)\rangle+\langle p-q,A(p-q)\rangle \geq \alpha |U_0(x,p)-U_0(y,q)|^2,\\
      \\
     \begin{array}{c}\langle G(x,p,u)-G(y,q,v),\phi(x)-\phi(y)\rangle+\langle F(x,p,u)D\phi(x)-F(y,q,v)D\phi(y),u-v\rangle\\+\langle 2A(b(x,p,u)-b(y,q,b)),p-q\rangle\\ \geq \alpha \left(|x-y|^2+|p-q|^2\right).
     \end{array}
\end{array}
\]
We make this remark as it leads to estimates for a wider class of functions $(G,F,b,U_0)$. 

\subsubsection{Return on the heuristic we presented and link with FBSDE}
We previously mentioned without detailing much that we could use a wider class of function $V$ satisfying
\begin{equation*}
\left\{
\begin{array}{c} \\
     \partial_t V+F(x,p,U)\cdot\nabla_x V+b(x,p,U)\cdot\nabla_p V-\sigma\Delta_p V=Q(x,p,U,V) \text{ in } (0,T)\times\reels^d\times \reels^m,\\
     V(0,x,p)=V_0(x,p)\subset \reels^m \text{ in } \R^d\times \R^m,
\end{array}
\right. 
\end{equation*}
to 'complete' the master equation, instead of a linear function $V(t,x,p)=Ap$, so long as conditions \eqref{V0 monotone} and \eqref{V monotone fgb} hold for this choice of $V$. If $Q$ is Lipschitz continuous, we can construct a Lispchitz solution $V$ to this equation which is well defined so long as $U$ is Lipschitz continuous. Using this fact the analysis we provided in the special case of a linear function can be extended to a wider class without much difficulty. In particular we stressed out that $V$ could be taken as an additional unknown instead of an explicit choice. We are now going to give one such example, which is based entirely on the notion of G-monotonicity \cite{doi:10.1137/S0363012996313549}, and give another interpretation of those wellposedness conditions

\begin{exemple}
    Let $N\in \mathcal{M}_d(\reels)$ and $M\in \mathcal{M}_{d\times m}(\reels)$. Looking at the equation satisfied by $(\tilde{U},V)$ for $\tilde{U}=N U$ and $V=MU$ and using the heuristic we developed, we expect that under the following two conditions:

\begin{equation*}
\begin{array}{c}
\forall (x,y,p,q,u,v)\in \Omega^2\times \reels^{2m}\times \reels^{2d},\\
\\
     \langle U_0(X)-U_0(Y),\Gamma (X-Y)\rangle\geq 0 \quad X=\left(\begin{array}{c} x  \\ p \end{array}\right) \quad Y=\left(\begin{array}{c} y  \\ q \end{array}\right) \quad \Gamma=(N^T,M^T),\\
     \\
     \langle \Gamma(\tilde{F}(X,u)-\tilde{F}(Y,v)),u-v\rangle+\langle \Gamma^T(G(X,u)-G(Y,v)),X-Y\rangle\geq \alpha |u-v|^2 \quad \tilde{F}=\left(\begin{array}{c} F  \\ b \end{array}\right).
\end{array}
\end{equation*}
we can show existence of a Lipschitz solution on any time interval. In fact, even though there seem to be few results in the PDE literature on the subject, this set of conditions is well known in the field of forward backward stochastic differential equations as G-monotonicity \cite{doi:10.1137/S0363012996313549}. 
\end{exemple}
We are not presenting this example to show that letting $V=Ap$ was a poor choice. Rather we want to stress the following: There is no one general best choice of a function V. \\Indeed we believe the conditions for G-monotonicity to hold are equally as strong as the one we previously required. However, depending on the nature of the coefficients, one or the other notion of monotonicity might hold. When $(F,b)$ depends on $U$ only as in Example \ref{b(U) inversible} G-monotonicity seems better suited to the study of equation \eqref{MFG b(x,p,u)} as there is no chance of the necessary assumptions for Theorem \ref{theorem:existence b(x,p,u)} to hold. However when $G$ does not depend on $p$  but the other terms do, G-monotonicity will never be satisfied. Point being that the way to approach equation \eqref{MFG b(x,p,u)} seems to depends heavily on the nature of the different coefficients. Some monotonicity assumptions on $U_0$ and $(F,G,b)$ does however seems unavoidable to propagate Lipschitz regularity.

\begin{remarque}
We do not want to delve too much in the link between forward backward stochastic differential equations and semi-linear systems of PDE as this is beyond the scope of this paper (for more on the subject, we recommend \cite{FBSDEpartang} ). Let us just state that the existence of a Lipschitz solution of \eqref{MFG b(x,p,u)} gives an existence result for the following FBSDE:
\begin{equation}
\label{mfg FBSDE}
\left\{
\begin{array}{l}
   \displaystyle X_t=x_0-\int_0^t F(X_s,p_s,V_s)ds,\\
   \displaystyle p_t=p_0-\int_0^t b(X_s,p_s,V_s)ds+\sqrt{2\sigma}W_t, \\
   \displaystyle V_t=U_0(X_T,p_T)+\int_t^T G(X_s,p_s,V_s)ds-\int_t^T Z_sdW_s \quad \forall t\in[0,T],
\end{array}
\right.
\end{equation}
for $(W_t)_{t\geq 0}$ a Brownian motion on $\reels^m$ and for any initial condition $(x_0,p_0)$. \\
Let 
\[
\left\{
\begin{array}{l}
   \displaystyle Y_t=x_0-\int_0^t F(Y_s,q_s,U(T-s,Y_s,q_s))ds,\\
   \displaystyle q_t=p_0-\int_0^t b(Y_s,q_s,U(T-s,Y_s,q_s))ds+\sqrt{2\sigma}W_t, \\
\end{array}
\right.
\]
for a Lipschitz solution U. Denoting $U_t=U(T-t,Y_t,q_t)$, we know that $(M_t)_{t\in[0,T]}$ defined by 
\[M_t=U_t+\int_0^t G(Y_s,q_s,U_s)ds\]
is a $\mathcal{F}-$martingale for the natural filtration $\mathcal{F}$ associated to the Brownian motion $(W_t)_{t\geq 0}$. Since $(F,G,U,b)$ are all Lipschitz continuous, $M$ is also a square integrable martingale. By the martingale representation theorem, there exists a predictable $\mathcal{F}-$adapted stochastic process $(C_s)_{s\in[0,T]}$ such that
\[
\begin{array}{l}
   \esp{ \displaystyle\int_0^T (C_s)^2ds}<+\infty,\\
   \displaystyle\forall t\leq T\quad  M_t=M_0+\int_0^t C_sdW_s.
\end{array}
\]
As a consequence:
\[M_T-\int_t^TC_sdW_s-\int_0^tG(Y_s,q_s,U_s)ds=U_t.\]
Which ends to show that $(Y_t,q_t,U_t,C_t)_{t\in[0,T]}$ is a solution of the previous FBSDE as by definition of $M$:
\[M_T=U_0(Y_T,q_T)+\int_0^T G(Y_s,q_s,U_s)ds.\]
Of course this is perfectly natural if we remember what is the master equation to the associated mean field game (which in this case take the form of the FBSDE \eqref{mfg FBSDE}). 
\end{remarque}

\nocite{*}
\printbibliography

\appendix

\appendixtitleon
\appendixtitletocon
\begin{appendices}
\section{A proof of Lemma \ref{equivalence alpha monotony}}
\equivalence*
\begin{proof}
$(ii)\implies (i)$ is trivial, as we only need to divide by $t^2$ and take the limit as $t\to 0$ for $x=y+t\xi$.

For the the other way, let us first assume that $f$ is invertible.
Its inverse $g$ satisfies (by taking $\xi=Dg\cdot \nu$)
\[\forall \nu \quad \nu\cdot D g \cdot\nu \geq \alpha |\nu|^2.\]
Let \[w(t)=\langle g(th+(1-t)z)-g(z),h-z\rangle-\alpha t|h-z|^2,\]
\[ \implies w'(t)=\langle h-z,Dg(th+(1-t)z),h-z\rangle-\alpha |h-z|^2\geq 0,\] by assumption. Because $w(0)=0$ we can deduce that 
\[w(1)=\langle g(h)-g(z),h-z\rangle - \alpha |h-z|^2\geq 0.\]
By taking $h=f(x)$ and $z=f(y)$ we conclude that for invertible $f$ we have 
\[\langle f(x)-f(y),x-y\rangle\geq \alpha |f(x)-f(y)|^2.\]
For non invertible $f$, because $(i)$ implies that $f$ is monotonous, we know that for any $\varepsilon$ $f^\varepsilon: x\to f(x)+\varepsilon x$ is invertible. Moreover:
\[\xi\cdot Df^\varepsilon\cdot \xi=\xi \cdot Df\cdot \xi+\varepsilon|\xi|^2\]
and 
\[|Df^\varepsilon\cdot \xi|^2=|D_xf\cdot \xi|^2+\varepsilon^2 |\xi^2|+2\varepsilon\xi\cdot Df\cdot \xi.\]
As a consequence, for $\varepsilon< \min(\frac{1}{2},\frac{1}{2\alpha})$ $f^\varepsilon$ satisfies property $(i)$ for $\alpha^\varepsilon=\alpha-2\alpha^2\varepsilon$, which means by what we just showed for invertible $f$ that 
\[\langle f^\varepsilon(x)-f^\varepsilon(y),x-y\rangle\geq (\alpha-2\alpha^2 \varepsilon) |f^\varepsilon(x)-f^\varepsilon(y)|^2.\]
We can then conclude by taking the limit as $\varepsilon\to 0$.
\end{proof}
\end{appendices}
\end{document}